\documentclass{article}

\usepackage{amsmath}
\usepackage{amsthm}
\usepackage{amssymb}
\usepackage{amsfonts}

\usepackage{subeqnarray}
\usepackage{bm}
\usepackage{tikz}
\usepackage{dsfont}
\usepackage{upgreek}
\usepackage{booktabs}
\usepackage{fancyhdr}
\usepackage{multirow}
\usepackage{multicol}
\usepackage{float}
\usepackage[inline]{enumitem}   
\usepackage{cases}
\usepackage[numbers]{natbib}

\usepackage[ruled,vlined,linesnumbered,inoutnumbered]{algorithm2e}

\usepackage{bbding}
\usepackage{caption}
\usepackage{todonotes}

\usepackage{mathrsfs}

\usepackage[top=2.5cm, bottom=2.5cm, left=3cm, right=3cm]{geometry} 

\usepackage{graphics,graphicx,epsf,epstopdf,subfigure}
\graphicspath{{./Figures/}}
\ifpdf
\DeclareGraphicsExtensions{.eps,.pdf,.png,.jpg}
\else
\DeclareGraphicsExtensions{.eps}
\fi

\usepackage[breaklinks,colorlinks,linkcolor=black,citecolor=black,urlcolor=black]{hyperref}

\providecommand{\U}[1]{\protect\rule{.1in}{.1in}}

\newtheorem{theorem}{Theorem}[section]

\newtheorem{lemma}[theorem]{Lemma}
\newtheorem{proposition}[theorem]{Proposition}
\newtheorem{definition}[theorem]{Definition}
\newtheorem{remark}{Remark}
\newtheorem{assumption}{Assumption}

\numberwithin{equation}{section}

\newcommand{\dist}{\mathrm{dist}}

\newcommand{\tr}{\mathrm{tr}}

\newcommand{\dom}{\mathrm{dom}}
\newcommand{\st}{\mathrm{s.\,t.}\,\,}

\newcommand{\bR}{\mathbb{R}}
\newcommand{\bN}{\mathbb{N}}

\newcommand{\cA}{\mathcal{A}}
\newcommand{\cB}{\mathcal{B}}
\newcommand{\cC}{\mathcal{C}}

\newcommand{\cO}{\mathcal{O}}
\newcommand{\cP}{\mathcal{P}}
\newcommand{\cR}{\mathcal{R}}

\newcommand{\cT}{\mathcal{T}}
\newcommand{\cU}{\mathcal{U}}

\newcommand{\tG}{\mathtt{G}}
\newcommand{\tV}{\mathtt{V}}
\newcommand{\tE}{\mathtt{E}}

\newcommand{\Rn}{\mathbb{R}^{n}}
  
\newcommand{\Rnp}{\mathbb{R}^{n \times p}} 
 
\newcommand{\Rpp}{\mathbb{R}^{p \times p}} 
\newcommand{\Rnn}{\mathbb{R}^{n \times n}}

\newcommand{\Rnq}{\mathbb{R}^{n \times q}}
\newcommand{\Rmq}{\mathbb{R}^{m \times q}}

\newcommand{\barX}{\bar{X}}

\newcommand{\bfone}{\mathbf{1}}
\newcommand{\bfzero}{\mathbf{0}}

\newcommand{\bfE}{\mathbf{E}}

\newcommand{\bfJ}{\mathbf{J}}

\newcommand{\bfU}{\mathbf{U}}
\newcommand{\bfV}{\mathbf{V}}
\newcommand{\bfW}{\mathbf{W}}
\newcommand{\bfX}{\mathbf{X}}

\newcommand{\bfZ}{\mathbf{Z}}

\newcommand{\Xik}{X_{i}^{(k)}}
\newcommand{\Xjk}{X_{j}^{(k)}}
\newcommand{\Xikn}{X_{i}^{(k+1)}}
\newcommand{\Xjkn}{X_{j}^{(k+1)}}

\newcommand{\Hik}{H_{i}^{(k)}}
\newcommand{\Hjk}{H_{j}^{(k)}}

\newcommand{\Uik}{U_{i}^{(k)}}
\newcommand{\Ujk}{U_{j}^{(k)}}
\newcommand{\Uikn}{U_{i}^{(k+1)}}

\newcommand{\Vik}{V_{i}^{(k)}}
\newcommand{\Vjk}{V_{j}^{(k)}}
\newcommand{\Vikn}{V_{i}^{(k+1)}}

\newcommand{\avXk}{\bar{\mathbf{X}}^{(k)}}
\newcommand{\avXkn}{\bar{\mathbf{X}}^{(k+1)}}

\newcommand{\avDk}{\bar{\mathbf{D}}^{(k)}}

\newcommand{\avGk}{\bar{\mathbf{G}}^{(k)}}

\newcommand{\avHk}{\bar{\mathbf{H}}^{(k)}}

\newcommand{\avUk}{\bar{\mathbf{U}}^{(k)}}
\newcommand{\avUkn}{\bar{\mathbf{U}}^{(k + 1)}}

\newcommand{\avVk}{\bar{\mathbf{V}}^{(k)}}
\newcommand{\avVkn}{\bar{\mathbf{V}}^{(k + 1)}}

\newcommand{\bfXk}{\mathbf{X}^{(k)}}
\newcommand{\bfXkn}{\mathbf{X}^{(k+1)}}

\newcommand{\bfZk}{\mathbf{Z}^{(k)}}
\newcommand{\bfZkn}{\mathbf{Z}^{(k+1)}}

\newcommand{\bfDk}{\mathbf{D}^{(k)}}
\newcommand{\bfDkn}{\mathbf{D}^{(k+1)}}

\newcommand{\bfGk}{\mathbf{G}^{(k)}}
\newcommand{\bfGkn}{\mathbf{G}^{(k+1)}}

\newcommand{\bfHk}{\mathbf{H}^{(k)}}

\newcommand{\bfUk}{\mathbf{U}^{(k)}}
\newcommand{\bfUkn}{\mathbf{U}^{(k + 1)}}

\newcommand{\bfVk}{\mathbf{V}^{(k)}}
\newcommand{\bfVkn}{\mathbf{V}^{(k + 1)}}

\newcommand{\barXk}{\bar{X}^{(k)}}
\newcommand{\barXkn}{\bar{X}^{(k + 1)}}

\newcommand{\barDk}{\bar{D}^{(k)}}

\newcommand{\barUk}{\bar{U}^{(k)}}

\newcommand{\barVk}{\bar{V}^{(k)}}
\newcommand{\barVkn}{\bar{V}^{(k + 1)}}

\newcommand{\barGk}{\bar{G}^{(k)}}

\newcommand{\barHk}{\bar{H}^{(k)}}

\newcommand{\zz}{^{\top}}
\newcommand{\inv}{^{-1}}
\newcommand{\ff}{_{\mathrm{F}}}
\newcommand{\fs}{^2_{\mathrm{F}}}

\newcommand{\SMnp}{\mathcal{S}_M^{n,p}}

\newcommand{\dkh}[1]{\left(#1\right)}
\newcommand{\hkh}[1]{\left\{#1\right\}}

\newcommand{\jkh}[1]{\left\langle#1\right\rangle}
\newcommand{\norm}[1]{\left\|#1\right\|}
\newcommand{\abs}[1]{\left\lvert #1\right\rvert}

\newcommand{\sym}{\mathrm{sym}}

\newcommand{\grad}{\mathrm{grad}\,}

\newcommand{\iid}{i \in [d]}

\newcommand{\iin}{i \in [n]}
\newcommand{\sumiid}{\sum_{i=1}^d}

\newcommand{\sumjjd}{\sum\limits_{j=1}^d}

\usepackage{hyperref}
\hypersetup{
    colorlinks=true,
    linkcolor=blue,
    filecolor=magenta,      
    urlcolor=blue,
    citecolor=blue,
    }

\allowdisplaybreaks
\graphicspath{ {./results/} }
\definecolor{Gray}{rgb}{0.5,0.5,0.5}

\makeatletter

\newcommand{\Rmnum}[1]{\expandafter\@slowromancap\romannumeral #1@}
\makeatother

\begin{document}

\title{A Double Tracking Method for Optimization with Decentralized Generalized Orthogonality Constraints}

\author{Lei Wang\thanks{Department of Statistics, Pennsylvania State University, University Park, PA, USA (\href{mailto:wlkings@lsec.cc.ac.cn}{wlkings@lsec.cc.ac.cn}).} 
\and Nachuan Xiao\thanks{Institute of Operational Research and Analytics, National University of Singapore, Singapore (\href{mailto:xnc@lsec.cc.ac.cn}{xnc@lsec.cc.ac.cn}).}
\and Xin Liu\thanks{State Key Laboratory of Scientific and Engineering Computing, Academy of Mathematics and Systems Science, Chinese Academy of Sciences, and University of Chinese Academy of Sciences, Beijing, China (\href{mailto:liuxin@lsec.cc.ac.cn}{liuxin@lsec.cc.ac.cn}).}}

\date{} 

\maketitle

\begin{abstract}
	
In this paper, we consider the decentralized optimization problems with generalized orthogonality constraints, where both the objective function and the constraint exhibit a distributed structure.
Such optimization problems, albeit ubiquitous in practical applications, remain unsolvable by existing algorithms in the presence of distributed constraints.
To address this issue, we convert the original problem into an unconstrained penalty model by resorting to the recently proposed constraint-dissolving operator.
However, this transformation compromises the essential property of separability in the resulting penalty function, rendering it impossible to employ existing algorithms to solve.
We overcome this difficulty by introducing a novel algorithm that tracks the gradient of the objective function and the Jacobian of the constraint mapping simultaneously.
The global convergence guarantee is rigorously established with an iteration complexity.
To substantiate the effectiveness and efficiency of our proposed algorithm, we present numerical results on both synthetic and real-world datasets.

\end{abstract}

% ---------------------------------------------------------------------------------------------------------------------------------

\section{Introduction}

Rapid advances in data collection and processing capabilities have paved the way for the utilization of distributed systems in a large number of practical applications, such as dictionary learning \cite{Raja2015cloud}, statistical inference \cite{Du2024empirical}, multi-agent control \cite{Dimarogonas2011distributed}, and neural network training \cite{Yuan2021decentlam,Zhang2024decentralized}.
The large scale and spatial/temporal disparity of data, coupled with the limitations in storage and computational resources, make centralized approaches infeasible or inefficient.
Consequently, decentralized algorithms are developed to solve an optimization problem through the collaboration of the agents, where the need to efficiently manage and process vast amounts of distributed data is paramount.

Given a distributed system of $d \in \bN_+$ agents connected by a communication network, the focus of this paper is on the following decentralized optimization problem with the generalized orthogonality constraint:
\begin{subequations} \label{opt:dest}
\begin{align}
	\label{opt:dest-obj}
	\min\limits_{X \in \Rnp} \hspace{2mm} 
	& f(X) :=  \sumiid f_i (X) \\
	\label{opt:dest-con}
	\st \hspace{3.5mm} & \sumiid X\zz M_i X = I_p,
\end{align}
\end{subequations}
where $f_i$ is a continuously differentiable local function,
$M_i \in \Rnn$ is a symmetric matrix,
and $I_p \in \Rpp$ denotes the $p \times p$ identity matrix.
Moreover, for each $i\in [d]$, both $f_i$ and $M_i$ are privately owned by agent $i$.
For convenience, we denote $M := \sum_{i = 1}^d M_i$, which is assumed to be positive definite throughout this paper. 
The feasible region of problem \eqref{opt:dest}, denoted by $\SMnp := \{X \in \Rnp \mid X\zz M X = I_p\}$, 
is an embedded submanifold of $\Rnp$ and commonly referred to as the generalized Stiefel manifold \cite{Absil2008}.
It is noteworthy that, different from classical decentralized optimization problems, the constraint \eqref{opt:dest-con} also has a distributed structure across agents, which leads to considerable challenges to solve \eqref{opt:dest} under the decentralized setting.
Throughout this paper, we make the following blanket assumption.

\begin{assumption} \label{asp:function}
	Each $f_i$ is continuously differentiable, and its gradient $\nabla f_i$ is locally Lipschitz continuous in $\Rn$.
\end{assumption}

Problems of the form \eqref{opt:dest} are found widely in many applications.
Below is a brief introduction to an important application of \eqref{opt:dest} in statistics.

\paragraph{Canonical Correlation Analysis (CCA)}

CCA is a fundamental and ubiquitous statistical tool 
that characterizes linear relationships between two sets of variables \cite{Hotelling1936relations}.
Let $A \in \Rnq$ and $B \in \Rmq$ be two datasets in the form of matrices,
where $n$ and $m$ are the dimensions of the two datasets respectively, 
and $q$ is the number of samples in each of the two datasets.
CCA aims to identify linear combinations of variables within each dataset to maximize their correlations.
Mathematically, CCA can be equivalently formulated as the following optimization problem \cite{Gao2023sparse},
\begin{equation} \label{opt:cca}
\begin{aligned}
	\min\limits_{X \in \bR^{(n + m) \times p}} \hspace{2mm} 
	& -\frac{1}{2} \tr \dkh{X\zz \Sigma X} \\
	\st \hspace{6mm} & X\zz M X = I_p,
\end{aligned}
\end{equation}
where $\Sigma \in \bR^{(n + m) \times (n + m)}$ and $M \in \bR^{(n + m) \times (n + m)}$ are two matrices generated by $A$ and $B$, and $\tr(\cdot)$ denotes the trace of a given square matrix.
Specifically, $\Sigma$ and $M$ have the following form,
\begin{equation*}
	\Sigma = \begin{bmatrix}
		AA\zz & AB\zz \\
		BA\zz & BB\zz
	\end{bmatrix}, \quad
	M = \begin{bmatrix}
		AA\zz & 0 \\
		0 & BB\zz
	\end{bmatrix},
\end{equation*}
respectively.

In this paper, we consider the distributed setting 
in which the samples contained in $A$ and $B$ are stored locally in $d$ locations, 
possibly having been collected and owned by different agents.
Suppose each agent $i$ possess $q_i$ samples and $q_1 + q_2 + \dotsb + q_d = q$.
Let $A_i \in \bR^{n \times q_i}$ and $B_i \in \bR^{m \times q_i}$ denote the local data of agent $i$. 
Then the data matrices $A$ and $B$ can be divided into $d$ blocks respectively,
namely, $A = [A_1 \; A_2 \; \dotsb \; A_d]$ and $B = [B_1 \; B_2 \; \dotsb \; B_d]$.
Under the aforementioned distributed setting, the optimization model \eqref{opt:cca} of CCA can be recast as the following form:
\begin{equation} \label{opt:cca-dist}
\begin{aligned}
	\min\limits_{X \in \bR^{(n + m) \times p}} \hspace{2mm} 
	& -\frac{1}{2} \sumiid \tr \dkh{X\zz \Sigma_i X} \\
	\st \hspace{6mm} & \sumiid X\zz M_i X = I_p,
\end{aligned}
\end{equation}
where $\Sigma_i \in \bR^{(n + m) \times (n + m)}$ and $M_i \in \bR^{(n + m) \times (n + m)}$ are given by
\begin{equation*}
	\Sigma_i = \begin{bmatrix}
		A_iA_i\zz & A_iB_i\zz \\
		B_iA_i\zz & B_iB_i\zz
	\end{bmatrix}, \quad
	M_i = \begin{bmatrix}
		A_iA_i\zz & 0 \\
		0 & B_iB_i\zz
	\end{bmatrix},
\end{equation*}
respectively.
There are other optimization problems in statistics with structures similar to CCA. 
Interested readers can refer to the references \cite{Chen2010coordinate,Gao2023sparse} for further details.

\subsection{Related Works}

Decentralized optimization has experienced significant advancements in recent decades, particularly in the Euclidean space. 
Various algorithms have been proposed to tackle different types of problems, such as gradient-based algorithms \cite{Nedic2009,Xu2015,Qu2017,Song2023optimal}, primal-dual frameworks \cite{Shi2015,Ling2015,Chang2015multi,Hajinezhad2019}, and second-order methods \cite{Bajovic2017newton,Zhang2021newton,Daneshmand2021newton}. 
In general, these algorithms are only capable of handling scenarios where variables are restricted in a convex subset of the Euclidean space. 
Consequently, these algorithms cannot be directly applied to solve \eqref{opt:dest-con}, where the feasible region is typically non-convex. 
Interested readers can refer to some recent surveys \cite{Nedic2018network,Chang2020distributed} for more comprehensive information.

In order to solve decentralized optimization problems on manifolds, many algorithms have adopted geometric tools derived from Riemannian optimization, including tangent spaces and retraction operators. 
For instance, the algorithms delineated in \cite{Chen2021decentralized,Deng2023decentralized} extend the Riemannian gradient descent method \cite{Absil2008} to the decentralized setting, which can be combined with gradient tracking techniques \cite{Sun2022distributed} to achieve the exact convergence.
Building on this foundation, \citet{Chen2023decentralized} further introduces the decentralized Riemannian conjugate gradient method.
Additionally, there are several algorithms devised to address nonsmooth optimization problems, such as subgradient algorithm \cite{Wang2023decentralized} and proximal gradient algorithm \cite{Wang2024decentralized}.
It is crucial to underscore that the computation of tangent spaces and retraction operators requires complete information about the matrix $M$, which is unattainable in a decentralized environment. 
This inherent limitation hinders the application of the aforementioned algorithms to the optimization problems with decentralized generalized orthogonality constraints.

There is another class of methodologies \cite{Wang2022decentralized,Wang2023variance,Wang2023smoothing,Sun2024global} that focuses on employing infeasible approaches, such as augmented Lagrangian methods, to tackle nonconvex manifold constraints.
These algorithms do not require each iterate to strictly adhere to manifold constraints, thereby allowing the pursuit of global consensus directly in the Euclidean space. 
Compared to Riemannian optimization methods, this type of algorithm requires only a single round of communication per iteration to guarantee their global convergence.
However, it is noteworthy that these algorithms are tailored for Stiefel manifolds and cannot handle scenarios where the constraints themselves exhibit a distributed structure.
Although we can draw inspiration from these algorithms to tackle the problem \eqref{opt:dest}, the resulting penalty model remains challenging to solve.
On the one hand, the penalty function usually forfeits the distributed structure inherent in the constraint, which is no longer separable across the agents.
On the other hand, without full knowledge of the matrix $M$, each agent can not independently compute its own local gradients.
Consequently, it is impossible to straightforwardly extend existing algorithms to solve the problem \eqref{opt:dest}.

\subsection{Contributions}

In decentralized optimization, current researches mainly focus on scenarios where the objective function exhibits a distributed structure. 
However, in problem \eqref{opt:dest}, the generalized orthogonal constraint also exhibits a similar structure, which triggers off an enormous difficulty in solving it. 

To develop a decentralized algorithm for the problem \eqref{opt:dest}, we employ the constraint dissolving operator introduced in \cite{Xiao2024dissolving} to construct an exact penalty model, which can not be solved by existing algorithms.
To efficiently minimize our proposed penalty model, we devise an approximate direction for the gradient of the penalty function, which is composed of separable components, including the gradient of the objective function and the Jacobian of the constraint mapping.
Then, we propose to track these two components simultaneously across the network to assemble them in the approximate direction. 
This double-tracking strategy is quite efficient to reach a consensus on the generalized Stiefel manifolds.

Based on the aforementioned techniques, we develop a novel constraint dissolving algorithm with double tracking (CDADT). 
To the best of our knowledge, this is the first algorithm capable of solving optimization problems with decentralized generalized orthogonality constraints.
Under rather mild conditions, we establish the global convergence of CDADT and provide its iteration complexity. 
Preliminary numerical results demonstrate the great potential of CDADT.

\subsection{Organization}

The rest of this paper is organized as follows.
Section \ref{sec:preliminaries} draws into some preliminaries related to the topic of this paper.
In Section \ref{sec:algorithm}, we develop a constraint dissolving algorithm with double tracking to solve the problem \eqref{opt:dest}.  
The convergence properties of the proposed algorithm are investigated in Section \ref{sec:convergence}.
Numerical results are presented in Section \ref{sec:experiments} to evaluate the performance of our algorithm.
Finally, this paper concludes with concluding remarks and key insights in Section \ref{sec:conclusion}.

\section{Preliminaries}

\label{sec:preliminaries}

In this section, we first present fundamental notations and network settings considered in this paper. 
Following this, we revisit the first-order stationarity condition of \eqref{opt:dest} and introduce the concepts of Kurdyka-{\L}ojasiewicz (K{\L}) property and constraint dissolving operator.

\subsection{Notations}

The following notations are adopted throughout this paper.
The Euclidean inner product of two matrices \(Y_1, Y_2\) with the same size is defined as \(\jkh{Y_1, Y_2}=\tr(Y_1\zz Y_2)\), where $\tr (B)$ stands for the trace of a square matrix $B$.
The $p \times p$ identity matrix is represented by $I_p \in \Rpp$.
We define the symmetric part of a square matrix $B$ as $\sym (B) := (B + B\zz) / 2$.
The Frobenius norm and 2-norm of a given matrix \(C\) are denoted by \(\norm{C}\ff\) and \(\norm{C}_2\), respectively. 
The $(i, j)$-th entry of a matrix $C$ is represented by $C (i, j)$.
The notations $\bfone_d \in \bR^d$ and $\bfzero_d \in \bR^d$ stand for the $d$-dimensional vector of all ones and all zeros, respectively.
The notations $\sigma_{\min} (X)$ and $\sigma_{\max} (X)$ represent the smallest and largest singular value of a matrix $X$, respectively.
The Kronecker product is denoted by $\otimes$.
For any $d \in \bN_+$, we denote $[d] := \{1, 2, \dotsc, d\}$.
We define the distance between a point $X$ and a set $\cC$ by $\dist(X, \cC) := \inf\{\norm{Y - X}\ff \mid Y \in \cC\}$.
Given a differentiable function $g(X) : \Rnp \to \bR$, the Euclidean gradient of $g$ with respect to $X$ is represented by $\nabla g(X)$.
Further notations will be introduced wherever they occur.

\subsection{Network Setting}

We assume that the $d$ agents are connected by a communication network.
And they can only exchange information with their immediate neighbors.
The network $\tG = (\tV, \tE)$ captures the communication links diffusing information among the agents.
Here, $\tV = [d]$ is composed of all the agents and $\tE = \{(i, j) \mid i \text{~and~} j \text{~are connected}\}$ represents the set of communication links.
Throughout this paper, we make the following assumptions on the network.

\begin{assumption} \label{asp:network}
	The communication network $\tG = (\tV, \tE)$ is connected.
	Furthermore, there exists a mixing matrix $W = [W(i, j)] \in \bR^{d \times d}$ associated with $\tG$ satisfying the following conditions.
	\begin{enumerate}
		
		\item $W$ is symmetric and nonnegative.
		
		\item $W \mathbf{1}_d = W\zz \mathbf{1}_d = \mathbf{1}_d$.
		
		\item $W(i, j) = 0$ if $i \neq j$ and $(i, j) \notin \tE$, and $W (i, j) > 0$ otherwise.
		
	\end{enumerate}
\end{assumption}

The conditions in Assumption \ref{asp:network} follow from standard assumptions in the literature \cite{Xiao2004,Nedic2018network}, which is dictated by the underlying network topology. 
According to the Perron-Frobenius Theorem \cite{Pillai2005perron}, we find that the eigenvalues of $W$ fall within the range $(-1, 1]$, and hence,
\begin{equation} \label{eq:lambda}
	\lambda := \norm{W - \bfone_d \bfone_d\zz / d}_2 < 1.
\end{equation}
The parameter $\lambda$ serves as a key indicator of the network connectivity and is instrumental in the analysis of decentralized methods.
Generally speaking, the closer $\lambda$ approaches $1$, the poorer the network connectivity becomes.

\subsection{Stationarity}

In this subsection, we delve into the first-order stationarity condition of the problem \eqref{opt:dest}.
Towards this end, we introduce some geometric concepts of Riemannian manifolds.
For each point $X \in \SMnp$, the tangent space to $\SMnp$ at $X$ is referred to as $\cT_{X} := \{D \in \Rnp \mid D\zz M X + X\zz M D = 0\}$. 
In this paper, we consider the Riemannian metric $\jkh{\cdot, \cdot}_M$ on $\cT_{X}$ that is induced from the inner product, i.e., $\jkh{V_1, V_2}_M = \jkh{V_1, M V_2} = \tr(V_1\zz M V_2)$.
The corresponding Riemannian gradient of a smooth function $f$ is given by
\begin{equation*}
	\grad f(X) := M\inv \nabla f(X) - X \sym (X\zz \nabla f(X)),
\end{equation*}
which is nothing but the projection of $\nabla f(X)$ onto $\cT_{X}$ under the metric $\jkh{\cdot, \cdot}_M$.
Finally, the first-order stationarity condition of the problem \eqref{opt:dest} can be stated as follows.

\begin{definition}
	A point $X \in \SMnp$ is called a first-order stationary point of the problem \eqref{opt:dest} if it satisfies the following condition,
	\begin{equation*}
		\grad f(X) = 0.
	\end{equation*}
\end{definition}

Since this paper focuses on infeasible algorithms for the problem \eqref{opt:dest}, we introduce the following definition of $\epsilon$-stationary point.

\begin{definition}
	A point $X \in \Rnp$ is called a first-order $\epsilon$-stationary point of the problem \eqref{opt:dest} if it satisfies the following condition,
	\begin{equation*}
		\max \hkh{\norm{\grad f (\cP (X))}\ff, \, \norm{X\zz M X - I_p}\ff} \leq \epsilon,
	\end{equation*}
	where $\cP (\cdot)$ is the projection operator onto the generalized Stiefel manifold $\SMnp$.
\end{definition}

\subsection{Kurdyka-{\L}ojasiewicz Property}

A part of the convergence results developed in this paper falls in the scope of a general class of functions that satisfy the Kurdyka-{\L}ojasiewicz (K{\L}) property \cite{Lojasiewicz1963propriete,Kurdyka1998gradients}.
Below, we introduce the basic elements to be used in the subsequent theoretical analysis.

For any $\tau > 0$, we denote by $\Phi_\tau$ the class of all concave and continuous functions $\phi:[0, \tau) \to \bR_+$ which satisfy the following conditions,
\begin{enumerate}

    \item $\phi (0) = 0$;

    \item $\phi$ is continuously differentiable on $(0, \tau)$ and continuous at $0$;

    \item $\phi^{\prime} (t) > 0$ for any $t \in (0, \tau)$.
    
\end{enumerate}
Now we define the K{\L} property.

\begin{definition}
    Let $g: \Rn \to (-\infty, +\infty]$ be a proper and lower semicontinuous function and $\partial g$ be the (limiting) subdifferential of $g$.
    \begin{enumerate}
    
        \item The function $g$ is said to satisfy the K{\L} property at $\bar{u} \in \dom (\partial g) := \{u \in \Rn \mid \partial g (u) \neq \emptyset\}$ if there exists a constant $\tau \in (0, +\infty]$, a neighborhood $\cU$ of $\bar{u}$ and a function $\phi \in \Phi_{\tau}$, such that for any $u \in \cU$ satisfying $g (\bar{u}) < g (u) < g (\bar{u}) + \tau$, the following K{\L} inequality holds,
        \begin{equation*}
            \phi^{\prime} (g (u) - g (\bar{u})) \, \dist (0, \partial g (u)) \geq 1.
        \end{equation*}
        The function $\phi$ is called a desingularizing function of $g$ at $\bar{u}$.

        \item We say $g$ is a K{\L} function if $g$ satisfies the K{\L} property at each point of $\dom (\partial g)$.
    \end{enumerate}
\end{definition}

K{\L} functions are ubiquitous in many practical applications, which covers a wealth of nonconvex nonsmooth functions.
For example, tame functions constitutes a wide class of K{\L} functions, including semialgebraic and real subanalytic functions.
We refer interested readers to \cite{Bolte2007lojasiewicz,Attouch2009convergence,Attouch2010proximal,Attouch2013convergence,Bolte2014proximal} for more details.

\subsection{Constraint Dissolving Operator}

The constraint dissolving operator proposed in \cite{Xiao2024dissolving} offers a powerful technique to handle manifold constraints, which can be leveraged to construct an unconstrained penalty model for Riemannian optimization problems.
Specifically, a constraint dissolving operator of the generalized orthogonality constraint \eqref{opt:dest-con} is given by
\begin{equation*}
	\cA(X) := \dfrac{1}{2} X \dkh{3 I_p -  \sumiid X\zz M_i X}.
\end{equation*}
Then solving the problem \eqref{opt:dest} can be converted into the unconstrained minimization of the following penalty function,
\begin{equation} \label{opt:dest-penalty}
	\min\limits_{X \in \Rnp} \hspace{2mm} 
	h(X) :=
	\sumiid f_i\dkh{\cA(X)}
	+ \dfrac{\beta}{4} \norm{\sumiid X\zz M_i X - I_p}\fs,
\end{equation} 
where $\beta > 0$ is a penalty parameter.

\citet{Xiao2024dissolving} have proved that \eqref{opt:dest} and \eqref{opt:dest-penalty} share the same first-order stationary points, second-order stationary points, and local minimizers in a neighborhood of $\SMnp$.
However, it is intractable to solve the problem \eqref{opt:dest-penalty} under the decentralized setting.
It is worth noting that, in the construction of the penalty function $h (X)$, the constraint \eqref{opt:dest-con} is integrated into the original objective function and the quadratic penalty term, resulting in the loss of the separable structure.
To the best of our knowledge, there are currently no algorithms equipped to tackle such problems. 
Therefore, in the next section, we will design a novel algorithm to solve the penalty model under the decentralized setting.

\section{Algorithm Development}

\label{sec:algorithm}

The purpose of this section is to develop an efficient decentralized algorithm to solve the penalty model \eqref{opt:dest-penalty}.
An approximate direction is first constructed for the gradient of the penalty function, which is easier to evaluate under the decentralized setting.
Then, we propose a double-tracking strategy to fabricate the approximate direction across the whole network.
The resulting algorithm is capable of reaching a consensus on the generalized Stiefel manifold.

\subsection{Gradient Approximation}

Under the conditions in Assumption \ref{asp:function}, the penalty function $h (X)$ in \eqref{opt:dest-penalty} is continuously differentiable, the gradient of which takes the following form:
\begin{equation*}
\begin{aligned}
	\nabla h(X) 
	= {} & \dfrac{1}{2} \sumiid \nabla f_i (Z) \mid_{Z = \cA(X)} \dkh{3 I_p - X\zz \sumiid M_i X}
	- \sumiid M_i X \sym\dkh{X\zz \sumiid \nabla f_i (Z) \mid_{Z = \cA(X)}} \\
	& + \beta \sumiid M_i X \dkh{X\zz \sumiid M_i X - I_p}.
\end{aligned}
\end{equation*}
Therefore, each agent $i$ can not compute the local gradient $\nabla f_i (Z) \mid_{Z = \cA (X)}$ individually since the evaluation of $\cA (X)$ requires the accessibility of $\{M_i\}$ over all the agents.
In light of the property that $\norm{\cA (X) - X}\ff = \cO (\norm{X\zz M X - I_p}\ff)$ whenever $X$ is not far away from $\SMnp$, we propose to approximate the local gradient $\nabla f_i (Z) \mid_{Z = \cA (X)}$ by $\nabla f_i (Z) \mid_{Z = X}$.
Hereafter, $\nabla f_i (Z) \mid_{Z = X}$ is denoted by $\nabla f_i (X)$ for simplicity.
As a result, we can obtain the following approximation of $\nabla h (X)$:
\begin{equation} \label{eq:search}
	H(X) =  S(X) + \beta Q(X),
\end{equation}
where
\begin{equation*}
	S(X) 
	= \dfrac{1}{2} \nabla f(X) \dkh{3 I_p - X\zz MX}
	- MX \sym\dkh{X\zz \nabla f(X)},
\end{equation*}
and
\begin{equation*}
	Q(X) = MX \dkh{X\zz MX - I_p}.
\end{equation*}
The approximate direction $H(X)$ of $\nabla h(X)$ possesses the following desirable properties.

\begin{lemma} \label{le:search}
	Let $\cR := \{X \in \Rnp \mid \|X\zz M X - I_p\|\ff \leq 1 / 6\}$
	be a bounded region 
	and $C_g := \sup_{X \in \cR} \norm{\nabla f (X)}\ff$ 
	be a positive constant.
	Then, if $\beta \geq \max\{(12 + 3 \sqrt{42} C_g) \sigma_{\min}^{-1/2} (M) / 5, \sigma_{\min}^{-3/2} (M) L_s^2\}$, we have
	\begin{equation*}
		\norm{H(X)}\fs 
		\geq \dfrac{1}{2} \sigma_{\min}^2 (M) \norm{\grad f(\cP (X))}\fs
		+ \beta \sigma_{\min}^{1/2} (M) \norm{X\zz M X - I_p}\fs,
	\end{equation*}
	for any $X \in \cR$.
\end{lemma}

\begin{proof}
	For any $X \in \cR$, we have the inequalities $\sigma_{\max}^2 (M^{1/2} X) \leq 7 / 6$ 
	and $\sigma_{\min}^2 (M^{1/2} X) \geq 5 / 6$, and hence,
	\begin{equation*}
	%\begin{aligned}
		\norm{M^{1/2} X \dkh{X\zz M X - I_p}}\fs 
		%= {} & \norm{M^{1/2} X \dkh{X\zz M X - I_p}}\fs \\
		\geq \sigma_{\min}^2 (M^{1/2} X) \norm{X\zz M X - I_p}\fs
		\geq \dfrac{5}{6} \norm{X\zz M X - I_p}\fs.
	%\end{aligned}
	\end{equation*}
	Then from the formulation of $S(X)$, it holds that 
	\begin{equation*}
	\begin{aligned}
		& \jkh{S(X), X \dkh{X\zz M X - I_p}} \\
		= {} & \dfrac{1}{2} \jkh{\nabla f (X) \dkh{3 I_p -  X\zz M X}, X \dkh{X\zz M X - I_p}}
		- \jkh{M X \sym \dkh{X\zz \nabla f (X)}, X \dkh{X\zz M X - I_p}} \\
		= {} & \dfrac{1}{2} \jkh{\sym \dkh{X\zz \nabla f (X)}, \dkh{X\zz M X - I_p} \dkh{3 I_p - X\zz M X} - 2 X\zz M X\dkh{X\zz M X - I_p}} \\
		= {} & - \dfrac{3}{2} \jkh{\sym \dkh{X\zz \nabla f (X)}, \dkh{X\zz M X - I_p}^2},
	\end{aligned}
\end{equation*}
	which implies that
	\begin{equation*}
	\begin{aligned}
		\abs{\jkh{S(X), X \dkh{X\zz M X - I_p}}}
		& \leq \dfrac{3}{2} \norm{\sym \dkh{X\zz \nabla f (X)}}\ff 
		\norm{\dkh{X\zz M X - I_p}^2}\ff \\
		& \leq \dfrac{3}{2} \norm{ (M^{1/2} X)\zz M^{-1/2} \nabla f (X)}\ff 
		\norm{\dkh{X\zz M X - I_p}^2}\ff \\
		& \leq \dfrac{\sqrt{42} C_g}{4 \sigma_{\min}^{1/2} (M)} \norm{X\zz M X - I_p}\fs.
	\end{aligned}
	\end{equation*}
	Now it can be readily verifies that
	\begin{equation} \label{eq:normhx}
	\begin{aligned}
		\norm{H (X)}\fs
		\geq {} & \sigma_{\min} (M) \norm{M^{-1/2} H (X)}\fs \\
		= {} & \sigma_{\min} (M) \norm{M^{-1/2} S(X)}\fs 
		+ 2 \beta \sigma_{\min} (M) \jkh{S(X), X \dkh{X\zz M X - I_p}} \\
		& + \beta^2 \sigma_{\min} (M) \norm{M^{1/2} X \dkh{X\zz M X - I_p}}\fs \\
%		\geq {} & \dfrac{\sigma_{\min} (M)}{\sigma_{\max} (M)} \norm{D (X)}\fs 
%		- \dfrac{\sqrt{42}}{2} \beta \sigma_{\min}^{1/2} (M) C_g \norm{X\zz M X - I_p}\fs \\
%		& + \dfrac{5}{6} \beta^2 \sigma_{\min} (M) \norm{X\zz M X - I_p}\fs \\
		\geq {} & \sigma_{\min}^2 (M) \norm{M\inv S(X)}\fs 
		+ \dfrac{1}{6} \dkh{5 \beta \sigma_{\min}^{1/2} (M) - 3 \sqrt{42} C_g} \beta \sigma_{\min}^{1/2} (M) \norm{X\zz M X - I_p}\fs \\
		\geq {} & \sigma_{\min}^2 (M) \norm{M\inv S(X)}\fs 
		+ 2 \beta \sigma_{\min}^{1/2} (M) \norm{X\zz M X - I_p}\fs,
	\end{aligned}
	\end{equation}
	where the last inequality follows from the condition $\beta \geq (12 + 3 \sqrt{42} C_g) \sigma_{\min}^{-1/2} (M) / 5$.
	
	Since it holds that $\sigma_{\min}^2 (M^{1/2} X) \geq 5 / 6$, we know that $X\zz M X$ is positive definite and $\cP (X) = X (X\zz M X)^{-1/2}$.
	Then straightforward calculations give rise to that
	\begin{equation*}
		X - \cP (X) = X (X\zz M X)^{-1/2} ((X\zz M X)^{1/2} + I_p)\inv (X\zz M X - I_p),
	\end{equation*}
	which further infers that
	\begin{equation*}
	\begin{aligned}
		\norm{X - \cP (X)}\ff
		\leq \sigma_{\min}^{-1/2} (M) \norm{X\zz M X - I_p}\ff.
	\end{aligned}
	\end{equation*}
	According to the local Lipschitz continuity of $S (X)$, there exists a constant $L_s > 0$ such that
	\begin{equation*}
	\begin{aligned}
		\norm{\grad f(\cP (X)) - M\inv S (X)}\ff
		= {} & \norm{M\inv S(\cP (X)) - M\inv S (X)}\ff \\
		\leq {} & \sigma_{\max} (M\inv) \norm{S(\cP (X)) - S (X)}\ff \\
		\leq {} & \sigma_{\min}\inv (M) L_s \norm{\cP (X) - X}\ff \\
		\leq {} & \sigma_{\min}^{-3/2} (M) L_s \norm{X\zz M X - I_p}\ff.
	\end{aligned}
	\end{equation*}
	Moreover, we have
	\begin{equation*}
	\begin{aligned}
		\norm{\grad f(\cP (X))}\fs
		\leq {} & 2 \norm{\grad f(\cP (X)) - M\inv S (X)}\fs
		+ 2 \norm{M\inv S (X)}\fs \\
		\leq {} & 2 \sigma_{\min}^{-3} (M) L_s^2 \norm{X\zz M X - I_p}\fs
		+ 2 \norm{M\inv S (X)}\fs.
	\end{aligned}
	\end{equation*}
	Combining the above relationship with \eqref{eq:normhx} yields that
	\begin{equation*}
	\begin{aligned}
		\norm{H (X)}\fs
		\geq {} & \dfrac{1}{2} \sigma_{\min}^2 (M) \norm{\grad f(\cP (X))}\fs
		+ \dkh{2 \beta \sigma_{\min}^{1/2} (M) - \sigma_{\min}^{-1} (M) L_s^2} \norm{X\zz M X - I_p}\fs \\
		\geq {} & \dfrac{1}{2} \sigma_{\min}^2 (M) \norm{\grad f(\cP (X))}\fs
		+ \beta \sigma_{\min}^{1/2} (M) \norm{X\zz M X - I_p}\fs,
	\end{aligned}
	\end{equation*}
	where the last inequality follows from the condition $\beta \geq \sigma_{\min}^{-3/2} (M) L_s^2$.
	The proof is completed.
\end{proof}

Lemma \ref{le:search} reveals that the norm of $\grad f (\cP (X))$ and the feasibility violation of $X$ are both controlled by the norm of $H(X)$ as long as $X \in \cR$ and $\beta$ is sufficiently large.
%And it is worthy of mentioning that $M\inv S(X) = \grad f(X)$ whenever $X \in \SMnp$. 
Therefore, as an approximation of $\nabla h(X)$, $H(X)$ can serve as a search direction to solve the problem \eqref{opt:dest-penalty}.

\subsection{Double Tracking Strategy}

In the construction of $H (X)$, each agent $i \in [d]$ is able to compute the local gradient $\nabla f_i (X)$ independently.
However, the evaluation of $H (X)$ still fails to be distributed into $d$ agents since it is not separable.
Nevertheless, we find that the components constituting $H (X)$ exhibit a separable structure as follows,
\begin{equation} \label{eq:track}
	H (X) 
	= \dfrac{d}{2} U (X) \dkh{3 I_p - d X\zz V (X)}
	- d^2 V (X) \sym\dkh{X\zz U (X)}
	+ \beta d V (X) \dkh{d X\zz V (X) - I_p},
\end{equation}
where
\begin{equation*}
	U (X) := \dfrac{1}{d} \sumiid \nabla f_i (X)
	\mbox{~and~}
	V (X) := \dfrac{1}{d} \sumiid M_i X. 
\end{equation*}
This observation inspires us to propose a double-tracking strategy. 
Specifically, we can first track $U (X)$ and $V (X)$ separately across the whole network by resorting to the dynamic average consensus \cite{Zhu2010discrete} protocol.
Then, these two components are collected together to assemble a global estimate of $H (X)$.
It is worth mentioning that $V (X)$ is exactly the Jacobian of the constraint mapping.

\subsection{Algorithm Description}

In this subsection, we describe the proposed algorithm to solve the problem \eqref{opt:dest-penalty}.
Hereafter, the notation $\Xik$ represents the $k$-th iterate of $X_i$.
Our algorithm introduces three auxiliary local variables $\Uik \in \Rnp$, $\Vik \in \Rnp$, and $\Hik \in \Rnp$ for each agent $i$ at the $k$-th iteration.
Specifically, $\Uik$ and $\Vik$ track $\nabla f (\Xik)$ and $M \Xik$ respectively, through the exchange of local information. In addition, $\Hik$ aims at estimating the search direction based on the formulation \eqref{eq:track}.
The key steps of our algorithm from the perspective of each agent are outlined below.

\paragraph{Step 1: Computing Search Direction.}
We first compute an approximate search direction based on \eqref{eq:track} as follows:
\begin{equation} \label{eq:update-h}
\begin{aligned}
	\Hik 
	= {} & \dfrac{d}{2} \Uik \dkh{3 I_p - d (\Xik)\zz \Vik}
	- d^2 \Vik \sym\dkh{(\Xik)\zz \Uik} \\
	& + \beta d \Vik \dkh{d (\Xik)\zz \Vik - I_p}.
\end{aligned}
\end{equation}

\paragraph{Step 2: Mixing Local Information.}
To ensure that the local estimates $X_i$’s asymptotically converge to a common value, we leverage the following consensus protocol.
Given the search directions $\Hik$'s in the previous step, we update the local variable $\Xikn$ by 
\begin{equation} \label{eq:update-x}
	\Xikn = \sumjjd W(i, j) \dkh{\Xjk - \eta \Hjk},
\end{equation}
where $\eta > 0$ is a stepsize.
The above procedure can be realized in a distributed manner.

\paragraph{Step 3: Tracking Gradient and Jacobian.}
Finally, to guarantee that each $\Uik$ and $\Vik$ track the average of $\nabla f_i (\Xik)$ and $M_i \Xik$ respectively, we leverage the dynamic average consensus \cite{Zhu2010discrete} technique.
The resulting gradient and Jacobian tracking schemes read as follows, 
\begin{align}
	\label{eq:update-g}
	& \Uikn = \sumjjd W(i, j) \dkh{\Ujk + \nabla f_j (\Xjkn) - \nabla f_j (\Xjk)}, \\ 
	\label{eq:update-j}
	& \Vikn = \sumjjd W(i, j) \dkh{\Vjk + M_j \Xjkn - M_j\Xjk},
\end{align}
with $U_i^{(0)} = \nabla f_i (X_i^{(0)})$ and $V_i^{(0)} = M_i X_i^{(0)}$.

Then based on the aforementioned steps, we formally present the detailed algorithmic framework in Algorithm \ref{alg:CDADT}, named {\it constraint dissolving algorithm with double tracking} and abbreviated to CDADT.

\begin{algorithm}[ht!]
	%\SetAlgoLined
	\caption{Constraint dissolving algorithm with double tracking (CDADT) for \eqref{opt:dest}.} 
	\label{alg:CDADT}
	
	\KwIn{$X_{\mathrm{init}} \in \Rnp$, , $\beta > 0$, and $\eta > 0$.}
	
	Set $k := 0$.
	
	\For{$\iid$}{
		
		Initialize $\Xik := X_{\mathrm{init}}$,
		$\Uik := \nabla f_i (X_{\mathrm{init}})$,
		and $\Vik := M_i X_{\mathrm{init}}$.
		
	}
	
	\While{``not converged''}
	{
		
		\For{$\iid$}{
			
			Compute $\Hik$ by \eqref{eq:update-h}.
			
			Update $\Xikn$ by \eqref{eq:update-x}.
			
			Update $\Uikn$ and $\Vikn$ by \eqref{eq:update-h} and \eqref{eq:update-j}, respectively.
			
		}
		
		Set $k := k + 1$.
		
	}
	
	\KwOut{$\{\Xik\}_{i = 1}^d$.}
	
\end{algorithm}

% \subsection{Compact Form}

In the rest of this subsection, we exhibit the compact form of Algorithm \ref{alg:CDADT}. For the sake of convenience, 
we denote $J = \bfone_d \bfone_d\zz / d \in \bR^{d \times d}$,
$\bfJ = J \otimes I_n \in \bR^{dn \times dn}$, 
$\bfE = \bfone_d \otimes I_n \in \bR^{dn \times n}$, 
and $\bfW = W \otimes I_n \in \bR^{dn \times dn}$.
It can be readily verified that $(\bfW - \bfJ) \bfJ = 0$.
The following notations are also used in the sequel.
\begin{itemize}
	
	\item $\bfXk = [(X_{1}^{(k)})\zz, \dotsc, (X_{d}^{(k)})\zz]\zz$,  
	$\barXk = \bfE\zz \bfXk / d$,
	$\avXk = \bfE \barXk = \bfJ \bfXk$.
	
	\item $\bfUk = [(U_{1}^{(k)})\zz, \dotsc, (U_{d}^{(k)})\zz]\zz$,
	$\barUk = \bfE\zz \bfUk / d$,
	$\avUk = \bfE \barUk = \bfJ \bfUk$.
	
	\item $\bfVk = [(V_{1}^{(k)})\zz, \dotsc, (V_{d}^{(k)})\zz]\zz$,
	$\barVk = \bfE\zz \bfVk / d$,
	$\avVk = \bfE \barVk = \bfJ \bfVk$.
	
	\item $\bfHk = [(H_{1}^{(k)})\zz, \dotsc, (H_{d}^{(k)})\zz]\zz$,
	$\barHk = \bfE\zz \bfHk / d$,
	$\avHk = \bfE \barHk = \bfJ \bfHk$.
	
	\item $\bfGk = [(\nabla f_1(X_1^{(k)}))\zz, \dotsc, (\nabla f_d(X_d^{(k)}))]\zz$,
	$\barGk = \bfE\zz \bfGk / d$,
	$\avGk = \bfE \barGk = \bfJ \bfGk$.
	
	\item $\bfDk = [(M_1 X_{1}^{(k)})\zz, \dotsc, (M_d X_{d}^{(k)})\zz]\zz$,
	$\barDk = \bfE\zz \bfDk / d$,
	$\avDk = \bfE \barDk = \bfJ \bfDk$.
	
\end{itemize}
By the formulation of the above notations, the main iteration loop of Algorithm \ref{alg:CDADT} 
can be summarized in the following compact form.
\begin{equation*}
\left\{
\begin{aligned}
	\bfXkn & = \bfW (\bfXk - \eta \bfHk), \\
	\bfUkn & = \bfW (\bfUk + \bfGkn - \bfGk), \\
	\bfVkn & = \bfW (\bfVk + \bfDkn - \bfDk).
\end{aligned}
\right.
\end{equation*}
Moreover, it is not difficult to check that, for any $k \in \bN$,
the following relationships hold.
\begin{equation}\label{eq:avg-iter}
	\barXkn = \barXk - \eta \barHk,~ 
	\barUk = \barGk,~
	\mbox{and~} \barVk = \barDk.
\end{equation}

\section{Convergence Analysis}

\label{sec:convergence}

In this section, we present the convergence analysis of Algorithm \ref{alg:CDADT}.
The global convergence guarantee is rigorously established under rather mild conditions, together with an iteration complexity.

\subsection{Consensus and Tracking Errors}
%disagreements deviation

This subsection is devoted to building the upper bound of consensus errors and tracking errors.
We start from the consensus error $\|\bfXkn - \avXkn\|\ff$ in the following lemma.

\begin{lemma} \label{le:consensus-error}
	Suppose the conditions in Assumption \ref{asp:network} hold.
	Then for any $k \in \bN$, it holds that
	\begin{equation*}
		\norm{\bfXkn - \avXkn}\fs
		\leq \dfrac{1 + \lambda^2}{2} \norm{\bfXk - \avXk}\fs
		+ \eta^2 C_1 \norm{\bfHk}\fs,
	\end{equation*}
	where $C_1 = \lambda^2 \dkh{1 + \lambda^2} / (1 - \lambda^2) > 0$ is a constant.
\end{lemma}

\begin{proof}
	By the update scheme in \eqref{eq:update-x} and straightforward calculations, we can attain that
	\begin{equation*}
		\begin{aligned}
			\norm{\bfXkn - \avXkn}\fs
			= & \norm{(\bfW - \bfJ) (\bfXk - \avXk) - \eta (\bfW - \bfJ) \bfHk}\fs \\
			\leq & \dkh{1 + \gamma} \norm{(\bfW - \bfJ) (\bfXk - \avXk)}\fs 
			+ \eta^2 \dkh{1 + 1 / \gamma} \norm{(\bfW - \bfJ) \bfHk}\fs \\
			\leq & \lambda^2 \dkh{1 + \gamma} \norm{\bfXk - \avXk}\fs
			+ \eta^2 \lambda^2 \dkh{1 + 1 / \gamma} \norm{\bfHk}\fs,
		\end{aligned}
	\end{equation*}
	where $\gamma = (1 - \lambda^2)/ (2 \lambda^2) > 0$ is a constant.
	This completes the proof 
	since $\lambda^2 \dkh{1 + \gamma} =  (1 + \lambda^2) / 2$
	and $\lambda^2 \dkh{1 + 1 / \gamma} = C_1$.
\end{proof}

Next, we proceed to bound the tracking errors $\|\bfUkn - \avUkn\|\ff$ and $\|\bfVkn - \avVkn\|\ff$.
To facilitate the narrative, we define the bounded region $\cB = \{X \in \Rnp \mid \|X\|\ff \leq C_x\}$, where $C_x = \sqrt{d} (1 +  \tilde{C}_x) > 0$ is a constant with $\tilde{C}_x = \sup \hkh{\|X\|\ff \mid X \in \cR} > 0$.

\begin{lemma} \label{le:tracking-error}
	Suppose the conditions in Assumptions \ref{asp:function} and \ref{asp:network} hold, $\Xikn \in \cB$ and $\Xik \in \cB$ for any $\iid$.
	Then we have
	\begin{equation*}
		\norm{\bfUkn - \avUkn}\fs
		\leq \dfrac{1 + \lambda^2}{2} \norm{\bfUk - \avUk}\fs
		+ 8 L_c^2 C_1 \norm{\bfXk - \avXk}\fs
		+ 2 \eta^2 L_c^2 C_1 \norm{\bfHk}\fs,
	\end{equation*}
	and
	\begin{equation*}
		\norm{\bfVkn - \avVkn}\fs
		\leq \dfrac{1 + \lambda^2}{2} \norm{\bfVk - \avVk}\fs
		+ 8 L_c^2 C_1 \norm{\bfXk - \avXk}\fs
		+ 2 \eta^2 L_c^2 C_1 \norm{\bfHk}\fs,
	\end{equation*}
	where $L_c = \max\{L_g, L_m\} > 0$ is a constant with $ L_m = \sup \hkh{\sigma_{\max}(M_i) \mid \iid} > 0$ and $L_g = \sup \hkh{\norm{\nabla f_i(X) - \nabla f_i(Y)}\ff / \norm{X - Y}\ff \mid X \neq Y, X \in \cB, Y \in \cB, \iid} > 0$.
\end{lemma}

\begin{proof}
	To begin with, straightforward  manipulations lead to that
	\begin{equation*}
		\begin{aligned}
			\norm{\bfUkn - \avUkn}\fs
			& = \norm{(\bfW - \bfJ) (\bfUk- \avUk) + (\bfW - \bfJ) (\bfGkn - \bfGk)}\fs \\
			& \leq \dkh{1 + \gamma} \norm{(\bfW - \bfJ) (\bfUk - \avUk)}\fs 
			+ \dkh{1 + 1 / \gamma} \norm{(\bfW - \bfJ) (\bfGkn - \bfGk)}\fs \\
			& \leq \lambda^2 \dkh{1 + \gamma} \norm{\bfUk - \avUk}\fs
			+ \lambda^2 \dkh{1 + 1 / \gamma} \norm{\bfGkn - \bfGk}\fs,
		\end{aligned}
	\end{equation*}
	where $\gamma = (1 - \lambda^2)/ (2 \lambda^2) > 0$ is a constant.
	According to the local Lipschitz continuity of $\nabla f_i$, it follows that
	\begin{equation*}
		\norm{\bfGkn - \bfGk}\ff \leq L_g \norm{\bfXkn - \bfXk}\ff.
	\end{equation*}
	Moreover, it can be readily verified that
    \begin{equation} \label{eq:xk+1-xk}
        \bfXkn - \bfXk
        = \bfW (\bfXk - \eta \bfHk) - \bfXk
        = (\bfW - I_{dn}) (\bfXk - \avXk) - \eta \bfW \bfHk,
	\end{equation}
    which implies that
	\begin{equation*}
        \norm{\bfXkn - \bfXk}\fs
        \leq 8 \norm{\bfXk - \avXk}\fs + 2 \eta^2 \norm{\bfHk}\fs.
	\end{equation*}
	Combining the above three relationships, we can obtain the bound for $\|\bfUkn - \avUkn\|\ff$. Furthermore, the bound for $\|\bfVkn - \avVkn\|\ff$ can be obtained by using the same technique, hence its proof is omitted for simplicity.  
\end{proof}

The following lemma demonstrates that $d \Uik$ and $d \Vik$ are estimates of $\nabla f (\barXk)$ and $M \barXk$ for each agent $i$, respectively.

\begin{lemma} \label{le:bar}
	Suppose that $\Xik \in \cB$ for any $\iid$.
	Then, under the conditions in Assumptions \ref{asp:function} and \ref{asp:network}, the following two inequalities hold 
	\begin{equation*}
		\norm{d \Uik - \nabla f (\barXk)}\fs 
		\leq 2 d^2 \norm{\Uik - \barUk}\fs 
		+ 2 d L_c^2 \norm{\bfXk - \avXk}\fs,
	\end{equation*}
	and
	\begin{equation*}
		\norm{d \Vik - M \barXk}\fs
		\leq 2 d^2 \norm{\Vik - \barVk}\fs 
		+ 2 d L_c^2 \norm{\bfXk - \avXk}\fs.
	\end{equation*}
\end{lemma}

\begin{proof}
	According to the local Lipschitz continuity of $\nabla f_i$, it follows that
	\begin{equation*}
	\begin{aligned}
		\norm{d \barUk - \nabla f (\barXk)}\fs
		= {} & \norm{d \barGk - \nabla f (\barXk)}\fs
		= \norm{\sumiid (\nabla f_i (\Xik) - \nabla f_i (\barXk))}\fs \\		
		\leq {} & d \sumiid \norm{\nabla f_i (\Xik) - \nabla f_i (\barXk)}\fs
		%\leq d L_g^2 \sumiid \norm{\Xik - \barXk}\fs
		\leq d L_g^2 \norm{\bfXk - \avXk}\fs,
	\end{aligned}
	\end{equation*}
	which further yields that
	\begin{equation*}
	\begin{aligned}
		\norm{d \Uik - \nabla f (\barXk)}\fs 
		= {} & \norm{d \Uik - d \barUk + d \barUk - \nabla f (\barXk)}\fs \\
		\leq {} & 2 d^2 \norm{\Uik - \barUk}\fs + 2 \norm{d \barUk - \nabla f (\barXk)}\fs \\
		\leq {} & 2 d^2 \norm{\Uik - \barUk}\fs + 2 d L_g^2 \norm{\bfXk - \avXk}\fs.
	\end{aligned}
	\end{equation*}
	Hence, we can conclude that the first assertion of this lemma holds.
	The second assertion can be proved by using a similar argument.
\end{proof}

We conclude this subsection by showing that $\barHk$ is an approximation of $H (\barXk)$ with the approximation error controlled by the consensus and tracking errors.
For convenience, we denote two constants $C_u = \sqrt{d} (1 + C_g) > 0$ and $C_v = \sqrt{d} (1 + C_x L_m) > 0$ to be used in the following lemma.

\begin{lemma} \label{le:hxbar}
	Let the conditions in Assumptions \ref{asp:function} and \ref{asp:network} hold.
	Suppose that $\|\Xik\|\ff \leq C_x$, $\|\Uik\|\ff \leq C_u$, and $\|\Vik\|\ff \leq C_v$ for any $\iid$.
	Then we have
	\begin{equation*}
	\begin{aligned}
		\norm{\barHk - H (\barXk)}\fs 
		\leq {} & \dfrac{C_2}{d} \norm{\bfUk - \avUk}\fs
		+ \dfrac{C_2 + C_3 \beta^2}{d} \norm{\bfVk - \avVk}\fs \\
		& + \dfrac{C_2 + C_3 \beta^2}{d} \norm{\bfXk - \avXk}\fs,
	\end{aligned}
	\end{equation*}
	where $C_2 > 0$ and $C_3 > 0$ are two constants.
\end{lemma}

\begin{proof}
	To begin with, we have
	\begin{equation*}
	\begin{aligned}
		\norm{\Hik - H(\barXk)}\fs
		\leq {} & \dfrac{3}{4} \norm{d \Uik (3 I_p - d (\Xik)\zz \Vik) - \nabla f(\barXk) (3 I_p - (\barXk)\zz M \barXk)}\fs \\
		& + 3 \norm{d^2 \Vik \sym ((\Xik)\zz \Uik) - M \barXk \sym ((\barXk)\zz \nabla f(\barXk)}\fs \\
		& + 3 \beta^2 \norm{d \Vik (d (\Xik)\zz \Vik - I_p) - M \barXk ((\barXk)\zz M \barXk - I_p)}\fs.
	\end{aligned}
	\end{equation*}
	By straightforward calculations, we can obtain that
	\begin{equation*}
	\begin{aligned}
		& \norm{d \Uik (3 I_p - d (\Xik)\zz \Vik) - \nabla f(\barXk) (3 I_p - (\barXk)\zz M \barXk)}\fs \\
		\leq {} & 3 \norm{(d \Uik - \nabla f (\barXk)) (3 I_p - d (\Xik)\zz \Vik)}\fs
		+ 3 \norm{\nabla f (\barXk) (\barXk - \Xik)\zz M \barXk}\fs \\
		& + 3 \norm{\nabla f (\barXk) (\Xik)\zz (M \barXk - d \Vik)}\fs \\
		\leq {} & 3 (18 p + 2 d^2 C_x^2 C_v^2) \norm{d \Uik - \nabla f (\barXk)}\fs 
		+ 3 \sigma_{\max}^2 (M) C_x^2 C_g^2 \norm{\Xik - \barXk}\fs \\
		& + 3 C_x^2 C_g^2\norm{d \Vik - M \barXk}\fs.
	\end{aligned}
	\end{equation*}
	And it can be readily verified that
	\begin{equation*}
	\begin{aligned}
		& \norm{d^2 \Vik \sym ((\Xik)\zz \Uik) - M \barXk \sym ((\barXk)\zz \nabla f(\barXk)}\fs \\
		\leq {} & 3 d^2 \norm{(d \Vik - M \barXk) (\Xik)\zz \Uik}\fs 
		+ 3 d^2 \norm{M \barXk (\Xik - \barXk)\zz \Uik}\fs \\
		& + 3 \norm{M \barXk (\barXk)\zz (d \Uik - \nabla f (\barXk))}\fs \\ 
		\leq {} & 3 d^2 C_x^2 C_u^2 \norm{d \Vik - M \barXk}\fs 
		+ 3 d^2 \sigma_{\max}^2 (M) C_x^2 C
		_u^2 \norm{\Xik - \barXk}\fs \\
		& + 3 \sigma_{\max}^2 (M) C_x^4 \norm{d \Uik - \nabla f (\barXk)}\fs.
	\end{aligned}
	\end{equation*}
	Moreover, we have
	\begin{equation*}
	\begin{aligned}
		& \norm{d \Vik (d (\Xik)\zz \Vik - I_p) - M \barXk ((\barXk)\zz M \barXk - I_p)}\fs \\
		\leq {} & 3 \norm{(d \Vik - M \barXk) (d (\Xik)\zz \Vik - I_p)}\fs 
		+ 3 \norm{M \barXk (\Xik)\zz (d \Vik - M \barXk)}\fs \\
		& + 3 \norm{M \barXk (\Xik - \barXk)\zz M \barXk}\fs \\
		\leq {} & 3 (2p + 2 d^2 C_x^2 C_v^2) \norm{d \Vik - M \barXk}\fs 
		+ 3 \sigma_{\max}^2 (M) C_x^4 \norm{d \Vik - M \barXk}\fs \\
		& + 3 \sigma_{\max}^4 (M) C_x^4 \norm{\Xik - \barXk}\fs.
	\end{aligned}
	\end{equation*}
	Combining the above three relationships, we can acquire that
	\begin{equation*}
	\begin{aligned}
		\norm{\Hik - H (\barXk)}\fs 
		\leq {} & C_{hu} \norm{d \Uik - \nabla f (\barXk)}\fs
		+ (C_{hv}^{\prime} + C_{hv}^{\prime \prime} \beta^2) \norm{d \Vik - M \barXk}\fs \\
		& + (C_{hx}^{\prime} + C_{hx}^{\prime \prime} \beta^2) \norm{\Xik - \barXk}\fs,
	\end{aligned}
	\end{equation*}
	where $C_{hu} = 9 (9 p + d^2 C_x^2 C_v^2 + 2 \sigma_{\max}^2 (M) C_x^4) / 2$,
	$C_{hv}^{\prime} = 9 C_x^2 C_g^2 / 4 + 9 d^2 C_x^2 C_u^2$,
	$C_{hv}^{\prime \prime} = 9 (2 p + 2 d^2 C_x^2 C_v^2 + \sigma_{\max}^2 (M) C_x^4)$,
	$C_{hx}^{\prime} = \sigma_{\max}^2 (M) C_{hv}^{\prime}$,
	and $C_{hx}^{\prime \prime} = 9 \sigma_{\max}^4 (M) C_x^4$ are five positive constants.
	For convenience, we further denote two constants $C_2 = \max\{1, 2 d^2 C_{hu}, 2 d^2 C_{hv}^{\prime}, d (C_{hx}^{\prime} + 2 d L_g^2 C_{hu}^{\prime} + 2 d L_m^2 C_{hv}^{\prime})\} \geq 1$ and $C_3 = \max\{2 d^2 C_{hv}^{\prime \prime}, d (C_{hx}^{\prime \prime} + 2 d L_m^2 C_{hv}^{\prime \prime})\} > 0$.
	Then according to Lemma \ref{le:bar}, it follows that
	\begin{equation} \label{eq:hxbar}
	\begin{aligned}
		\norm{\Hik - H (\barXk)}\fs 
		\leq {} & 2 d^2 C_{hu} \norm{\Uik - \barUk}\fs
		+ 2 d^2 (C_{hv}^{\prime} + C_{hv}^{\prime \prime} \beta^2) \norm{\Vik - \barVk}\fs \\
		& + (C_{hx}^{\prime} + 2 d L_g^2 C_{hu}^{\prime} + 2 d L_m^2 C_{hv}^{\prime} + (C_{hx}^{\prime \prime} + 2 d L_m^2 C_{hv}^{\prime \prime}) \beta^2) \norm{\bfXk - \avXk}\fs \\
		\leq {} & C_2 \norm{\Uik - \barUk}\fs
		+ (C_2 + C_3 \beta^2) \norm{\Vik - \barVk}\fs \\
		& + \dfrac{1}{d} (C_2 + C_3 \beta^2) \norm{\bfXk - \avXk}\fs.
	\end{aligned}
	\end{equation}
	The last thing to do in the proof is to show that
	\begin{equation*}
		\norm{\barHk - H (\barXk)}\fs
		\leq \dfrac{1}{d} \sumiid \norm{\Hik - H (\barXk)}\fs.
	\end{equation*}
	Combining the above two relationships, we complete the proof.
\end{proof}

\subsection{Boundedness of Iterates}

In this subsection, we aim to show that the iterate sequence generated by Algorithm \ref{alg:CDADT} is restricted in a neighborhood of the feasible region. 
Moreover, the average of local variables is always restricted in the bounded region $\cR$, which guarantees the usage of Lemma \ref{le:search}.

We first prove the following technical lemma.

\begin{lemma} \label{le:omega}
	Suppose that $\barXkn$ is generated by \eqref{eq:avg-iter}
	with $\barXk \in \cR$ and $ \norm{\barHk - H (\barXk)}\fs \leq 3$.
	Let the penalty parameter $\beta$ and stepsize $\eta$ satisfy
	\begin{equation*}
		\beta \geq \sqrt{\dfrac{864 (3 + C_s^2)}{\sigma_{\min} (M)}},
	\end{equation*}
	and
	\begin{equation*}
		0 < \eta \leq \min\hkh{
			\dfrac{1}{4 L_q \beta}, \, 
			\dfrac{3}{5 \sigma_{\min} (M) \beta}, \,
			\dfrac{\beta}{480(3 + C_s^2)}
		},
	\end{equation*}
	respectively.
	Then, under the conditions in Assumptions \ref{asp:function} and \ref{asp:network}, we have $\barXkn \in \cR$.
\end{lemma}

\begin{proof}
	It follows from the relationship \eqref{eq:avg-iter} that 
	\begin{equation*}
	\begin{aligned}
		\barXkn 
		= \barXk - \eta H (\barXk) + \eta (H (\barXk) - \barHk)
		%= \barXk - \eta \beta M \barXk ((\barXk)\zz M \barXk - I_p) + \eta Y^{(k)}
		= \barXk - \eta \beta Q (\barXk) + \eta Y^{(k)}, \\
	\end{aligned}
	\end{equation*}
	where $Y^{(k)} := H (\barXk) - \barHk - S (\barXk)$ and $Y^{(k)}$ satisfies
	\begin{equation*}
		\norm{Y^{(k)}}\fs \leq 2 \norm{H (\barXk) - \barHk}\fs
		+ 2 \norm{S (\barXk)}\fs
		\leq 2 (3 + C_s^2).
	\end{equation*}
	Since $\barXk \in \cR$, we have $\sigma_{\max}^2 (M^{1/2} \barXk) \leq 7 / 6$ 
	and $\sigma_{\min}^2 (M^{1/2} \barXk) \geq 5 / 6$, and hence,
	\begin{equation} \label{eq:qbarxk}
	\begin{aligned}
		\norm{Q (\barXk)}\fs 
		= {} & \norm{M^{1/2} M^{1/2} \barXk ((\barXk)\zz M \barXk - I_p)}\fs \\
		\geq {} & \sigma_{\min}^2 (M^{1/2}) \sigma_{\min}^2 (M^{1/2} \barXk) \norm{(\barXk)\zz M \barXk - I_p}\fs \\
		\geq {} & \dfrac{5}{6} \sigma_{\min} (M) \norm{(\barXk)\zz M \barXk - I_p}\fs.
	\end{aligned}
	\end{equation}
	By virtue of the Young's inequality, we can obtain that
	\begin{equation*}
	\begin{aligned}
		\norm{\barXkn - \barXk}\fs
		= {} & \eta^2 \norm{\beta Q (\barXk) - Y^{(k)}}\fs
		\leq 2 \eta^2 \beta^2 \norm{Q (\barXk)}\fs + 2 \eta^2 \norm{Y^{(k)}}\fs \\
		\leq {} & \dfrac{\eta \beta}{2 L_q} \norm{Q (\barXk)}\fs + \dfrac{\eta}{2 L_q \beta} \norm{Y^{(k)}}\fs,
	\end{aligned}
	\end{equation*}
	where the last inequality results from the condition that $\eta \leq 1 / (4 L_q \beta)$.
	Moreover, we have
	\begin{equation*}
	\begin{aligned}
		\jkh{Q (\barXk), \barXkn - \barXk}
		= {} & - \eta \beta \norm{Q (\barXk)}\fs
		+ \eta \jkh{Q (\barXk), Y^{(k)}} \\
		\leq {} & - \dfrac{3\eta \beta}{4} \norm{Q (\barXk)}\fs
		+ \dfrac{\eta}{\beta} \norm{Y^{(k)}}\fs.
	\end{aligned}
	\end{equation*}
	For convenience, we denote $c(X) = \|X\zz M X - I_p\|\fs / 4$. 
	According to the local Lipschitz continuity of $\nabla c(X) = Q(X)$, there exists a constant $L_q > 0$ such that
	\begin{equation*}
	\begin{aligned}
		c (\barXkn) 
		\leq {} & c (\barXk) 
		+ \jkh{Q (\barXk), \barXkn - \barXk} 
		+ \dfrac{L_q}{2} \norm{\barXkn - \barXk}\fs \\
		\leq {} & c (\barXk) - \dfrac{\eta \beta}{2} \norm{Q (\barXk)}\fs
		+ \dfrac{5 \eta}{4 \beta} \norm{Y^{(k)}}\fs,
	\end{aligned}
	\end{equation*}
	which together with \eqref{eq:qbarxk} infers that
	\begin{equation*}
		\norm{(\barXkn)\zz M \barXkn - I_p}\fs 
		\leq \dkh{1 - \dfrac{5}{3} \sigma_{\min} (M) \eta \beta} \norm{(\barXk)\zz M \barXk - I_p}\fs
		 + \dfrac{5 \eta}{\beta} \norm{Y^{(k)}}\fs.
	\end{equation*}
	Now we investigate the above relationship in the following two cases.
	
	\paragraph{Case I: $\norm{(\barXk)\zz M \barXk - I_p}\ff \leq 1/ 12$.} 
	Since $\eta \leq \min\{3 / (5 \sigma_{\min} (M) \beta), \beta / (480(3 + C_s^2))\}$, we have
	%$\eta \leq (4 + 3M_g)^{-1} / 18$ and $\eta \leq (1 + M_g)^{-1} / 6$,
	\begin{equation*}
		\norm{(\barXkn)\zz M \barXkn - I_p}\fs
		\leq \norm{(\barXk)\zz M \barXk - I_p}\fs 
		+ \dfrac{1}{48} 
		= \dfrac{1}{36}.
	\end{equation*}
	
	\paragraph{Case II: $\norm{(\barXk)\zz M \barXk - I_p}\ff > 1/ 12$.}
	It can be readily verified that
	\begin{equation*}
		\begin{aligned}
			\norm{(\barXkn)\zz M \barXkn - I_p}\fs - \norm{(\barXk)\zz M \barXk - I_p}\fs
			\leq {} & - \dfrac{5}{3} \sigma_{\min} (M) \eta \beta \norm{(\barXk)\zz M \barXk - I_p}\fs \\
			& + \dfrac{10 \eta}{\beta} (3 + C_s^2) \\
			\leq {} & 5 \eta \dkh{- \dfrac{1}{432} \sigma_{\min} (M) \beta + \dfrac{2}{\beta} (3 + C_s^2)} \\
			\leq {} & 0,
		\end{aligned}
	\end{equation*}
	where the last inequality follows from the conditions $\beta^2 \geq 864 (3 + C_s^2) / \sigma_{\min} (M)$.
	Hence, we arrive at 
	\begin{equation*}
		\norm{(\barXkn)\zz M \barXkn - I_p}\ff 
		\leq \norm{(\barXk)\zz M \barXk - I_p}\ff 
		\leq 1 / 6.
	\end{equation*}
	Combining the above two cases together, we complete the proof.
\end{proof}

Based on Lemma \ref{le:omega},
we can prove the main results of this subsection.

\begin{proposition} \label{prop:bound}
	Suppose the conditions in Assumptions \ref{asp:function} and \ref{asp:network} hold.
	Let the penalty parameter $\beta$ and stepsize $\eta$ satisfy
	\begin{equation} \label{eq:beta-b}
		\beta \geq \max\hkh{
			1, \,
			%2 L_c, \,
			2 L_c \sqrt{C_2}, \,
			\sqrt{\dfrac{864 (3 + C_s^2)}{\sigma_{\min} (M)}}, \,
			%\sqrt{\dfrac{32 L_c^2 C_1}{1 - \lambda^2}}, \,
			\sqrt{\dfrac{32 L_c^2 C_1 C_2}{1 - \lambda^2}}
		},
	\end{equation}
	and
	\begin{equation} \label{eq:eta-b}
		\eta \leq \min\hkh{
			\dfrac{1}{4 L_q \beta}, \, 
			\dfrac{3}{5 \sigma_{\min} (M) \beta}, \,
			\dfrac{\beta}{480(3 + C_s^2)}, \,
			\dfrac{1}{\beta (C_h^{\prime} + C_h^{\prime \prime} \beta)}\sqrt{\dfrac{d (1 - \lambda^2)}{2 C_1 (C_2 + C_3 \beta^2)}}
		},
	\end{equation}
	respectively.
	Then for any $k \in \bN$, it holds that 
	\begin{equation} \label{eq:bound-x}
	\left\{
	\begin{aligned}	
		& \barXk \in \cR,~
		\norm{\bfXk}\ff \leq C_x,~ 
		\norm{\bfUk}\ff \leq C_u,~
		\norm{\bfVk}\ff \leq C_v, \\
		& \norm{\bfXk - \avXk}\fs \leq \dfrac{d}{\beta^2 (C_2 + C_3 \beta^2)}, \\
		& \norm{\bfUk - \avUk}\fs \leq \dfrac{d}{C_2},~
		\norm{\bfVk - \avVk}\fs \leq \dfrac{d}{C_2 + C_3 \beta^2}.
	\end{aligned}
	\right.
	\end{equation}
\end{proposition}

\begin{proof}
	We intend to prove this proposition by mathematical induction.
	The argument \eqref{eq:bound-x} directly holds at iteration $k = 0$ resulting from the initialization. 
	Now, we assume that this argument holds at iteration $k$, and investigate the situation at iteration $k + 1$.
	
	To begin with, it follows from Lemma \ref{le:hxbar} and the condition $\beta \geq 1$ that
	\begin{equation*}
	\begin{aligned}
		\norm{\barHk - H (\barXk)}\fs 
		\leq \dfrac{C_2}{d} \norm{\bfUk - \avUk}\fs
		+ \dfrac{C_2 + C_3 \beta^2}{d} \dkh{\norm{\bfVk - \avVk}\fs + \norm{\bfXk - \avXk}\fs}
		\leq 3.
	\end{aligned}
	\end{equation*}
	Hence, according to Lemma \ref{le:omega}, we know that $\barXkn \in \cR$.
	Moreover, we have
	\begin{equation*}
	\begin{aligned}
		\norm{\Hik}\ff
		\leq {} & \dfrac{d}{2} \norm{\Uik (3 I_p - d (\Xik)\zz \Vik)}\ff
		+ d^2 \norm{\Vik \sym ((\Xik)\zz \Uik)}\ff \\
		& + \beta d \norm{\Vik (d (\Xik)\zz \Vik - I_p)}\ff \\ 
		\leq {} & \dfrac{C_h^{\prime} + C_h^{\prime \prime} \beta}{\sqrt{d}}, 
	\end{aligned}
	\end{equation*}
	where $C_h^{\prime} = 3 d^{3/2} C_u (d C_x C_v + \sqrt{p}) / 2 > 0$ and $C_h^{\prime \prime} = d^{3/2} C_v (d C_x C_v + \sqrt{p}) > 0$ are two constants.
	Then it can be readily verified that
	\begin{equation*}
		\norm{\bfHk}\fs 
		= \sumiid \norm{\Hik}\fs
		%\leq \sumiid \dfrac{(C_h^{\prime} + C_h^{\prime \prime} \beta)^2}{d}
		\leq (C_h^{\prime} + C_h^{\prime \prime} \beta)^2.
	\end{equation*}
	Combining Lemma \ref{le:consensus-error} with the condition $\eta^2 \leq \dfrac{d (1 - \lambda^2)}{2 \beta^2 C_1 (C_2 + C_3 \beta^2) (C_h^{\prime} + C_h^{\prime \prime} \beta)^2}$ leads to that
	\begin{equation*}
	\begin{aligned}
		\norm{\bfXkn - \avXkn}\fs
		\leq {} & \dfrac{1 + \lambda^2}{2} \norm{\bfXk - \avXk}\fs
		+ \eta^2 C_1 \norm{\bfHk}\fs \\
		\leq {} & \dfrac{d (1 + \lambda^2)}{2 \beta^2 (C_2 + C_3 \beta^2)}
		+ \eta^2 C_1 (C_h^{\prime} + C_h^{\prime \prime} \beta)^2 \\
		\leq {} & \dfrac{d}{\beta^2 (C_2 + C_3 \beta^2)},
	\end{aligned}
	\end{equation*}
	which together with the condition $\beta \geq 1$ implies that
	\begin{equation*}
	\begin{aligned}
		\norm{\bfXkn}\ff 
		\leq \norm{\bfXkn - \avXkn}\ff
		+ \sqrt{d} \norm{\barXkn}\ff
		\leq \sqrt{d} (1 +  \tilde{C}_x)
		= C_x.
	\end{aligned}
	\end{equation*}
	As a direct consequence of Lemma \ref{le:tracking-error}, we can proceed to show that
	\begin{equation*}
	\begin{aligned}
		\norm{\bfVkn - \avVkn}\fs
		\leq {} & \dfrac{1 + \lambda^2}{2} \norm{\bfVk - \avVk}\fs
		+ 8 L_c^2 C_1 \norm{\bfXk - \avXk}\fs
		+ 2 \eta^2 L_c^2 C_1 \norm{\bfHk}\fs \\
		\leq {} & \dfrac{d (1 + \lambda^2)}{2 (C_2 + C_3 \beta^2)}
		+ \dfrac{8 d L_c^2 C_1}{\beta^2 (C_2 + C_3 \beta^2)}
		+ 2 \eta^2 L_c^2 C_1 (C_h^{\prime} + C_h^{\prime \prime} \beta)^2 \\
		\leq {} & \dfrac{d}{C_2 + C_3 \beta^2},
	\end{aligned}
	\end{equation*}
	where the last inequality results from the conditions $\eta^2 \leq \dfrac{d (1 - \lambda^2)}{8 L_c^2 C_1 (C_2 + C_3 \beta^2) (C_h^{\prime} + C_h^{\prime \prime} \beta)^2}$ and $\beta^2 \geq \dfrac{32 L_c^2 C_1}{1 - \lambda^2}$.
	Furthermore, since $C_2 + C_3 \beta^2 \geq C_2 \geq 1$, we have
	\begin{equation*}
	\begin{aligned}
		\norm{\bfVkn}\ff 
		\leq \norm{\bfVkn - \avVkn}\ff 
		+ \sqrt{d} \norm{\barVkn}\ff
		\leq \sqrt{d} (1 + C_x L_m)
		= C_v.  
	\end{aligned}
	\end{equation*}
	Similarly, under the conditions $\eta^2 \leq \dfrac{d (1 - \lambda^2)}{8 L_c^2 C_1 C_2 (C_h^{\prime} + C_h^{\prime \prime} \beta)^2}$ and $\beta^2 \geq \dfrac{32 L_c^2 C_1 C_2}{1 - \lambda^2}$, we can show that
	\begin{equation*}
		\norm{\bfUkn - \avUkn}\fs \leq \dfrac{d}{C_2},~
		\mbox{and~}
		\norm{\bfUkn}\ff \leq \sqrt{\dfrac{d}{C_2}} + \sqrt{d} C_g 
		\leq C_u. 
	\end{equation*}
	The proof is completed.
\end{proof}

\subsection{Sufficient Descent}

The purpose of this subsection is to evaluate the descent property of the sequence $\{ h(\barXk)\}$.

\begin{lemma} \label{le:des-h}
	Suppose that Assumptions \ref{asp:function} and \ref{asp:network} hold.
	Let all the conditions in Proposition \ref{prop:bound} be satisfied.
	We further assume that $\eta \leq 1 / (8 (L_s + L_q \beta))$.
	Then for any $k \in \bN$, it holds that
	\begin{equation*}
	\begin{aligned}
		h (\barXkn) 
		\leq {} & h (\barXk) 
		- \dfrac{5}{8} \eta \norm{H (\barXk)}\fs
		+ \dfrac{9}{4} \eta \norm{\barHk - H (\barXk)}\fs \\
		& + 4 \eta L_g^2 C_4 \norm{(\barXk)\zz M \barXk - I_p}\fs,
	\end{aligned}
	\end{equation*}
	where $C_4 > 0$ is a constant.
\end{lemma}

\begin{proof}
	According to Proposition \ref{prop:bound}, 
	we know that the inclusion $\barXk \in \cR$ 
	holds for any $k \in \bN$.
	Since $\nabla h$ is locally Lipschitz continuous, there exist two constants $L_s > 0$ and $L_q > 0$ such that
	\begin{equation*}
	\begin{aligned}
		h (\barXkn) 
		= {} & h (\barXk - \eta \barHk)
		\leq h (\barXk) - \eta \jkh{\nabla h (\barXk), \barHk} 
		+ \dfrac{1}{2} \eta^2 \dkh{L_s + L_q \beta} \norm{\barHk}\fs \\
%		= {} & h (\barXk) 
%		- \eta \jkh{\nabla h(\barXk) - H(\barXk), \barHk}
%		- \eta \jkh{H(\barXk), \barHk} \\
%		& + \dfrac{1}{2} \eta^2 \dkh{L_s + \beta L_q} \norm{\barHk}\fs \\
		= {} & h (\barXk) 
		- \eta \norm{H(\barXk)}\fs 
		- \eta \jkh{\nabla h(\barXk) - H(\barXk), \barHk} \\
		& - \eta \jkh{H(\barXk), \barHk - H(\barXk)} 
		+ \dfrac{1}{2} \eta^2 \dkh{L_s + L_q \beta} \norm{\barHk}\fs \\
		\leq {} & h (\barXk) 
		- \dfrac{7}{8} \eta \norm{H(\barXk)}\fs 
		- \eta \jkh{\nabla h(\barXk) - H(\barXk), \barHk} \\
		& - \eta \jkh{H(\barXk), \barHk - H(\barXk)} 
		+ \dfrac{1}{8} \eta \norm{\barHk - H (\barXk)}\fs,
	\end{aligned}
	\end{equation*}
	where the last inequality follows from the condition $\eta \leq 1 / (8 (L_s + L_q \beta))$ and the relationship
	\begin{equation*}
		\norm{\barHk}\fs
		\leq 2 \norm{\barHk - H (\barXk)}\fs
		+ 2 \norm{H (\barXk)}\fs.
	\end{equation*}
	By virtue of the Young's inequality, we can obtain that
	\begin{equation*}
	\begin{aligned}
		\abs{\jkh{\nabla h(\barXk) - H(\barXk), \barHk}}
		\leq {} & 4 \norm{\nabla h(\barXk) - H(\barXk)}\fs
		+ \dfrac{1}{16} \norm{\barHk}\fs \\
		\leq {} & 4 L_g^2 C_4 \norm{(\barXk)\zz M \barXk - I_p}\fs
		+ \dfrac{1}{8} \norm{\barHk - H (\barXk)}\fs \\
		& + \dfrac{1}{8} \norm{H (\barXk)}\fs,
	\end{aligned}
	\end{equation*}
	where $C_4 = (7 \sigma_{\min} (M) (144p + 1) + 343 \sigma_{\max} (M)) / (432 \sigma_{\min}^2 (M)) > 0$ is a constant.
	Moreover, we have 
	\begin{equation*}
	\begin{aligned}
		\abs{\jkh{H (\barXk), \barHk - H (\barXk)}}
		\leq {} & 2 \norm{\barHk - H (\barXk)}\fs 
		+ \dfrac{1}{8} \norm{H (\barXk)}\fs.
	\end{aligned}
	\end{equation*}
	Combing the above relationships, 
	we finally arrive at the assertion of this lemma. %by combining the above two relationships.
\end{proof}

The above lemma indicates that the sequence $\{ h(\barXk)\}$ is not necessarily decreasing in a monotonic manner. 
To address this issue, we introduce the following merit function,
\begin{equation*}
	\hbar (\bfX, \bfU, \bfV)
	:= h (\bfE\zz \bfX / d) 
	+ \norm{(I_{dn} - \bfJ) \bfX}\fs
	+ \rho \norm{(I_{dn} - \bfJ) \bfU}\fs
	+ \rho \norm{(I_{dn} - \bfJ) \bfV}\fs,
\end{equation*}
where $\rho = (1 - \lambda^2) / (128 L_c^2 C_1) > 0$ is a constant.
For convenience, we denote $\hbar^{(k)} := \hbar (\bfXk, \bfUk, \bfVk)$ hereafter.
The following proposition illustrates that the sequence $\{\hbar^{(k)}\}$ satisfies a sufficient descent property, and hence, is monotonically decreasing.

\begin{proposition} \label{prop:des-L}
	Suppose that Assumptions \ref{asp:function} and \ref{asp:network} hold.
	Let all the conditions in Proposition \ref{prop:bound} be satisfied.
	We further assume that
	\begin{equation} \label{eq:beta-d}
		\beta \geq \max \hkh{
			\dfrac{12 + 3 \sqrt{42} C_g}{5 \sigma_{\min}^{1/2} (M)}, \,
			\dfrac{L_s^2}{\sigma_{\min}^{3/2} (M)}, \,
			\dfrac{16 L_g^2 C_4}{\sigma_{\min}^{1/2} (M)}
		},
	\end{equation}
	and
	\begin{equation} \label{eq:eta-d}
		\eta \leq \min \hkh{
			\dfrac{1}{16 d C_5}, \,
			\dfrac{1}{8 (L_s + L_q \beta)}, \,
			\dfrac{d (1 - \lambda^2) \min\{1, 2 \rho\}}{36 (C_2 + C_3 \beta^2)}, \,
			\sqrt{\dfrac{(1 - \lambda^2) \min\{1, 2\rho\}}{32 (C_2 + C_3 \beta^2) C_5}}
		}.
	\end{equation}
	Then for any $k \in \bN$, the following sufficient descent property holds, namely,
	\begin{equation*} %\label{eq:des-merit}
	\begin{aligned}
		\hbar^{(k + 1)}
		\leq {} & \hbar^{(k)}
		- \dfrac{1}{4} \eta \sigma_{\min}^2 (M) \norm{\grad f(\cP (\barXk))}\fs
		- \dfrac{1}{4} \eta \beta \sigma_{\min}^{1/2} (M) \norm{(\barXk)\zz M \barXk - I_p}\fs \\
		& - \dfrac{1 - \lambda^2}{4} \norm{\bfXk - \avXk}\fs
		- \dfrac{(1 - \lambda^2) \rho}{4} \norm{\bfUk - \avUk}\fs
		- \dfrac{(1 - \lambda^2) \rho}{4} \norm{\bfVk - \avVk}\fs.
	\end{aligned}
	\end{equation*}	
\end{proposition}

\begin{proof}
	Combining Lemmas \ref{le:hxbar} and \ref{le:des-h} gives rise to that
	\begin{equation*}
	\begin{aligned}
		h (\barXkn) 
		\leq {} & h (\barXk) 
		- \dfrac{5}{8} \eta \norm{H (\barXk)}\fs
		+ 4 \eta L_g^2 C_4 \norm{(\barXk)\zz M \barXk - I_p}\fs \\
		& + \dfrac{9 \eta C_2}{4 d} \norm{\bfUk - \avUk}\fs
		+ \dfrac{9 \eta (C_2 + C_3 \beta^2)}{4 d} \norm{\bfVk - \avVk}\fs \\
		& + \dfrac{9 \eta (C_2 + C_3 \beta^2)}{4 d} \norm{\bfXk - \avXk}\fs \\
		\leq {} & h (\barXk) 
		- \dfrac{5}{8} \eta \norm{H (\barXk)}\fs
		+ 4 \eta L_g^2 C_4 \norm{(\barXk)\zz M \barXk - I_p}\fs \\
		& + \dfrac{(1 - \lambda^2) \rho}{8} \norm{\bfUk - \avUk}\fs
		+ \dfrac{(1 - \lambda^2) \rho}{8} \norm{\bfVk - \avVk}\fs \\
		& + \dfrac{1 - \lambda^2}{16} \norm{\bfXk - \avXk}\fs,
	\end{aligned}
	\end{equation*}
	where the last inequality follows from the condition $\eta \leq d (1 - \lambda^2) \min\{1, 2 \rho\} / (36 (C_2 + C_3 \beta^2))$.
	Then according to Lemmas \ref{le:consensus-error} and \ref{le:tracking-error}, it follows that
	\begin{equation} \label{eq:hbarkn}
	\begin{aligned}
		\hbar^{(k + 1)}
		\leq {} & \hbar^{(k)}
		- \dfrac{5}{8} \eta \norm{H (\barXk)}\fs
		+ \eta^2 C_5 \norm{\bfHk}\fs
		+ 4 \eta L_g^2 C_4 \norm{(\barXk)\zz M \barXk - I_p}\fs \\
		& - \dfrac{3 (1 - \lambda^2) \rho}{8} \norm{\bfUk - \avUk}\fs
		- \dfrac{3 (1 - \lambda^2) \rho}{8} \norm{\bfVk - \avVk}\fs \\
		& - \dfrac{5 (1 - \lambda^2)}{16} \norm{\bfXk - \avXk}\fs,
	\end{aligned}
	\end{equation}
	where $C_5 = C_1 + 4 L_c^2 C_1 \rho > 0$ is a constant.
	As a direct consequence of the relationship \eqref{eq:hxbar}, we can proceed to show that
	\begin{equation*}
	\begin{aligned}
		\norm{\bfHk}\fs 
		= {} & \sumiid \norm{\Hik}\fs
		\leq 2 d \norm{H (\barXk)}\fs
		+ 2 \sumiid \norm{\Hik - H (\barXk)}\fs \\
		\leq {} & 2 d \norm{H (\barXk)}\fs
		+ 2 C_2 \norm{\bfUk - \avUk}\fs
		+ 2 (C_2 + C_3 \beta^2) \norm{\bfVk - \avVk}\fs \\
		& + 2 (C_2 + C_3 \beta^2) \norm{\bfXk - \avXk}\fs.
	\end{aligned}
	\end{equation*}
	By virtue of the condition $\eta \leq \min\{1 / (16 d C_5), \sqrt{(1 - \lambda^2) \min\{1, 2\rho\} / (32 (C_2 + C_3 \beta^2) C_5)}\}$, we have
	\begin{equation} \label{eq:hk}
	\begin{aligned}
		\norm{\bfHk}\fs 
		\leq {} & \dfrac{1}{8 \eta C_5} \norm{H (\barXk)}\fs
		+ \dfrac{(1 - \lambda^2) \rho}{8 \eta^2 C_5} \norm{\bfUk - \avUk}\fs
		+ \dfrac{(1 - \lambda^2) \rho}{8 \eta^2 C_5} \norm{\bfVk - \avVk}\fs \\
		& + \dfrac{1 - \lambda^2}{16 \eta^2 C_5} \norm{\bfXk - \avXk}\fs.
	\end{aligned}
	\end{equation}
	Then, by combining two relationships \eqref{eq:hbarkn} and \eqref{eq:hk}, it can be readily verified that
	\begin{equation*}
	\begin{aligned}
		\hbar^{(k + 1)}
		\leq {} & \hbar^{(k)}
		- \dfrac{\eta}{2} \norm{H (\barXk)}\fs
		- \dfrac{(1 - \lambda^2) \rho}{4} \norm{\bfUk - \avUk}\fs
		- \dfrac{(1 - \lambda^2) \rho}{4} \norm{\bfVk - \avVk}\fs \\
		& - \dfrac{1 - \lambda^2}{4} \norm{\bfXk - \avXk}\fs
		+ 4 \eta L_g^2 C_4 \norm{(\barXk)\zz M \barXk - I_p}\fs,
	\end{aligned}
	\end{equation*}
	which together with Lemma \ref{le:search} yields that
	\begin{equation*}
	\begin{aligned}
		\hbar^{(k + 1)}
		\leq {} & \hbar^{(k)}
		 - \dfrac{\eta}{4} \sigma_{\min}^2 (M) \norm{\grad f(\cP (\barXk))}\fs
		 - \dfrac{\eta}{2} \dkh{\beta \sigma_{\min}^{1/2} (M) - 8  L_g^2 C_4} \norm{(\barXk)\zz M \barXk - I_p}\fs \\
		& - \dfrac{1 - \lambda^2}{4} \norm{\bfXk - \avXk}\fs
		- \dfrac{(1 - \lambda^2) \rho}{4} \norm{\bfUk - \avUk}\fs
		- \dfrac{(1 - \lambda^2) \rho}{4} \norm{\bfVk - \avVk}\fs.
	\end{aligned}
	\end{equation*}	
	Since $\beta \geq 16  \sigma_{\min}^{-1/2} (M)  L_g^2 C_4$, we can obtain the desired sufficient descent property.
	The proof is completed.
\end{proof}

\subsection{Global Convergence}

Based on the sufficient descent property of $\{\hbar^{(k)}\}$, we can finally establish the global convergence guarantee of Algorithm \ref{alg:CDADT} to a first-order stationary point of the problem \eqref{opt:dest}.
Moreover, the iteration complexity is also presented.

\begin{theorem} \label{thm:global}
	Suppose Assumptions \ref{asp:function} and \ref{asp:network} hold.
	Let the penalty parameter $\beta$ satisfy the conditions \eqref{eq:beta-b} and \eqref{eq:beta-d} and the stepsize $\eta$ satisfy the conditions \eqref{eq:eta-b} and \eqref{eq:eta-d}.
	Then the sequence $\{\bfXk\}$ has at least one accumulation point.
	Moreover, for any accumulation point $\bfX^{\ast}$, there exists a first-order stationary point $\barX^{\ast} \in \SMnp$ of the problem \eqref{opt:dest} such that $\bfX^{\ast} = (\bfone_d \otimes I_n) \barX^{\ast}$.
	Finally, the following relationships hold, namely,
	\begin{equation} \label{eq:sublinear-gradient}
		\min_{k = 0, 1, \dotsc, K - 1} \norm{\grad f (\cP (\barXk))}\fs
		\leq \dfrac{4 (\hbar^{(0)} - \underline{h})}{\eta \sigma_{\min}^2 (M) K},
	\end{equation}
	\begin{equation} \label{eq:sublinear-consensus}
		\min_{k = 0, 1, \dotsc, K - 1} \norm{\bfXk - \avXk}\fs
		\leq \dfrac{4 (\hbar^{(0)} - \underline{h})}
		{(1 - \lambda^2) K},
	\end{equation}
	\begin{equation} \label{eq:sublinear-feasible}
		\min_{k = 0, 1, \dotsc, K - 1} \norm{(\barXk)\zz M \barXk - I_p}\fs
		\leq \dfrac{4 (\hbar^{(0)} - \underline{h})}{\eta \beta \sigma_{\min}^{1/2} (M) K},
	\end{equation}
	where $\underline{h}$ is a constant.
\end{theorem}

\begin{proof}
	According to Proposition \ref{prop:bound}, we know that the sequence $\{\bfXk\}$ is bounded.
	Then the lower boundedness of $\{h (\barXk)\}$
	is owing to the continuity of $h$.
	Hence, there exists a constant $\underline{h}$ such that
	\begin{equation*}
		\hbar^{(k)} \geq h (\barXk) \geq \underline{h},
	\end{equation*}
	for any $k \in \bN$.
	It follows from Proposition \ref{prop:des-L} that
	the sequence $\{\hbar^{(k)}\}$ is convergent and 
	the following relationships hold,	
	\begin{equation} \label{eq:limit}
		\lim\limits_{k \to \infty} \norm{\grad f (\cP(\barXk))}\fs = 0, \,
		\lim\limits_{k \to \infty} \norm{\bfXk - \avXk}\fs = 0, \,
		\lim\limits_{k \to \infty} \norm{(\barXk)\zz M \barXk - I_p}\fs = 0.
	\end{equation}
	According to the Bolzano-Weierstrass theorem, it follows that $\{\bfXk\}$ exists an accumulation point, say $\bfX^{\ast}$.
	Then the relationships in \eqref{eq:limit} imply that there exists a first-order stationary point $\barX^{\ast} \in \SMnp$ of the problem \eqref{opt:dest} such that $\bfX^{\ast} = (\bfone_d \otimes I_n) \barX^{\ast}$.
	
	The last thing to do in the proof is to show that the relationships \eqref{eq:sublinear-gradient}-\eqref{eq:sublinear-feasible} hold.
	Indeed, it follows from Proposition \ref{prop:des-L} that
	\begin{equation*}
	\begin{aligned}
		\sum_{k = 0}^{K - 1} \norm{\grad f (\cP(\barXk))}\fs 
		\leq \dfrac{4}{\eta \sigma_{\min}^2 (M)} \sum_{k = 0}^{K - 1} \dkh{\hbar^{(k)} - \hbar^{(k + 1)}}
		\leq \dfrac{4 (\hbar^{(0)} - \hbar^{(K)})}{\eta \sigma_{\min}^2 (M)}
		\leq \dfrac{4 (\hbar^{(0)} - \underline{h})}{\eta \sigma_{\min}^2 (M)},
	\end{aligned}
	\end{equation*}
	which yields the relationship \eqref{eq:sublinear-gradient}.
	The other relationships can be proved similarly.
	Therefore, we complete the proof.
\end{proof}

The global sublinear convergence rate in Theorem \ref{thm:global} guarantees that Algorithm \ref{alg:CDADT} is able to return a first-order $\epsilon$-stationary point in at most $\cO (\epsilon^{-2})$ iterations.
Since Algorithm \ref{alg:CDADT} performs three rounds of communication per iteration, the total number of communication rounds required to obtain a first-order $\epsilon$-stationary point is also $\cO (\epsilon^{-2})$ at the most.

\begin{remark}
	Under the conditions in Theorem \ref{thm:global}, the tracking errors also asymptotically converges to zero at a sublinear rate as follows,
	\begin{equation*} %\label{eq:sublinear-tracking}
		\min_{k = 0, 1, \dotsc, K - 1} \max \hkh{\norm{\bfUk - \avUk}\fs, \, \norm{\bfVk - \avVk}\fs}
		\leq \dfrac{4 (\hbar^{(0)} - \underline{h})}
		{(1 - \lambda^2) \rho K}.
	\end{equation*}
\end{remark}

\subsection{Convergence Under Kurdyka-{\L}ojasiewicz Property}

In this subsection, we establish the convergence of Algorithm \ref{alg:CDADT} when the problem \eqref{opt:dest} satisfies the K{\L} property. 
In particular, for any $i \in [d]$, we prove that the entire sequence $\{X_{i}^{(k)}\}_{k \in \bN}$ converges to a stationary point of the problem \eqref{opt:dest} with guaranteed asymptotic convergence rates.

To begin with, we define the following quantity,
\begin{equation*}
\begin{aligned}
    r^{(k)} := {} & \norm{\grad f(\cP (\barXk))}\ff
	+ \norm{(\barXk)\zz M \barXk - I_p}\ff 
	+ \norm{\bfXk - \avXk}\ff \\
	& + \norm{\bfUk - \avUk}\ff
	+ \norm{\bfVk - \avVk}\ff.
\end{aligned}
\end{equation*}
Then it can be readily verified from Proposition \ref{prop:des-L} that
\begin{equation} \label{eq:des-L}
    \hbar^{(k)} - \hbar^{(k + 1)} \geq C_e (r^{(k)})^2,
\end{equation}
where $C_e := \min\{\eta \sigma_{\min}^2 (M), \eta \beta \sigma^{1/2} (M), 1 - \lambda^2, (1 - \lambda^2) \rho\} / 20 $ is a prefixed positive constant.

Next, we give a lower bound of $r^{(k)}$ by the norm of $\nabla \hbar := (\nabla_{\bfX} \hbar, \nabla_{\bfU} \hbar, \nabla_{\bfV} \hbar)$ in the following lemma.
\begin{lemma} \label{le:nabla-h}
	Suppose Assumption \ref{asp:function} and Assumption \ref{asp:network} hold.
	Let all the conditions in Theorem \ref{thm:global} be satisfied.
	Then, for any $k \in \bN$, there holds that
	\begin{equation*}
		\norm{\nabla \hbar(\bfXk, \bfUk, \bfVk)}\ff \leq C_r r^{(k)},
	\end{equation*}
	where $C_r > 0$ is a constant. 
\end{lemma}

\begin{proof}
	To begin with, from straightforward calculations, we can obtain that 
	\begin{equation*}
	\begin{aligned}
		\nabla_{\bfX} \hbar (\bfXk, \bfUk, \bfVk) 
		= {} & \bfE \nabla h (\barXk) / d + 2 (I_{dn} - \bfJ) (\bfXk - \avXk) \\
		= {} & \bfE M \grad f (\cP (\barXk)) / d + \bfE (\nabla h (\barXk) - M \grad f (\cP (\barXk)))/ d \\
		& + 2 (I_{dn} - \bfJ) (\bfXk - \avXk).
	\end{aligned}
	\end{equation*}
	Notice that $M \grad f (\cP (\barXk)) = S (\cP (\barXk))$ and $H (\barXk) = S (\barXk) + \beta Q (\barXk)$, we have
	\begin{equation}
        \label{Eq:le:nabla-h_0}
	\begin{aligned}
		\norm{\nabla h (\barXk) - M \grad f (\cP (\barXk))}\ff
		\leq {} & \norm{\nabla h (\barXk) - H (\barXk)}\ff 
		+ \norm{S (\barXk) - S (\cP (\barXk))}\ff \\
		& + \beta \norm{Q (\barXk)}\ff \\
		\leq {} & C_6 \norm{(\barXk)\zz M \barXk - I_p}\ff
		+ \beta \norm{Q (\barXk)}\ff,
	\end{aligned}
	\end{equation}
	where $C_6 := \sqrt{C_4} L_g + \sigma_{\min}^{-1/2} (M) L_s > 0$. 
	According to Proposition \ref{prop:bound}, it follows that $\barXk \in \cR$, and hence,
	\begin{equation}
        \label{Eq:le:nabla-h_1}
		\norm{M \barXk ((\barXk)\zz M \barXk - I_p)}\ff
		% \leq \sigma_{\max} (M^{1 / 2}) \norm{M^{1 / 2} \barXk ((\barXk)\zz M \barXk - I_p)}\ff
		\leq \sqrt{\dfrac{7}{6}} \sigma_{\max}^{1 / 2} (M) \norm{(\barXk)\zz M \barXk - I_p}\ff.
	\end{equation}
	Combining \eqref{Eq:le:nabla-h_0} with \eqref{Eq:le:nabla-h_1}, we can acquire that
	\begin{equation*}
	\begin{aligned}
		\norm{\nabla h (\barXk) - M \grad f (\cP (\barXk))}\ff
		\leq \dkh{C_6 + \sqrt{\dfrac{7}{6}} \sigma_{\max}^{1 / 2} (M) \beta} \norm{(\barXk)\zz M \barXk - I_p}\ff,
	\end{aligned}
	\end{equation*}
	which further implies that
	\begin{equation}
        \label{Eq:le:nabla-h_2}
	\begin{aligned}
		\norm{\nabla_{\bfX} \hbar (\bfXk, \bfUk, \bfVk)}\ff
		\leq {} & \dfrac{\sqrt{d}}{d} \sigma_{\max} (M) \norm{\grad f (\cP (\barXk))}\ff + 2 \norm{\bfXk - \avXk}\ff\\
		& + \dfrac{\sqrt{d}}{d} \dkh{C_6 + \sqrt{\dfrac{7}{6}} \sigma_{\max}^{1 / 2} (M) \beta} \norm{(\barXk)\zz M \barXk - I_p}\ff.
	\end{aligned}
	\end{equation}
	Using a similar argument, we can proceed to prove that
	\begin{equation}
            \label{Eq:le:nabla-h_3}
		\norm{\nabla_{\bfU} \hbar (\bfXk, \bfUk, \bfVk)}\ff 
		\leq 2 \rho \norm{\bfUk - \avUk}\ff, \\
	\end{equation}
	and
	\begin{equation}
            \label{Eq:le:nabla-h_4}
		\norm{\nabla_{\bfV} \hbar (\bfXk, \bfUk, \bfVk)}\ff \leq 2 \rho \norm{\bfVk - \avVk}.
	\end{equation}
    Then from \eqref{Eq:le:nabla-h_2}-\eqref{Eq:le:nabla-h_4}, we have
    \begin{equation*}
        \norm{\nabla \hbar(\bfXk, \bfUk, \bfVk)}\ff \leq (3 \sigma_{\max}^2 (M) / d + 3 (C_6 + \sqrt{7 / 6} \sigma_{\max}^{1/2} (M) \beta)^2 / d + 8 \rho^2 + 12)^{1/2} r^{(k)}.
    \end{equation*}
    This completes the proof with $C_r$ chosen as $(3 \sigma_{\max}^2 (M) / d + 3 (C_6 + \sqrt{7 / 6} \sigma_{\max}^{1/2} (M) \beta)^2 / d + 8 \rho^2 + 12)^{1/2}$. 
\end{proof}

Let $\bfZk := (\bfXk, \bfUk, \bfVk)$. The following lemma reveals that the distance between $\bfZkn$ and $\bfZk$ can be controlled by $r^{(k)}$.

\begin{lemma} \label{le:zk+1-zk}
    Suppose Assumption \ref{asp:function} and Assumption \ref{asp:network} hold. 
    Then with the same conditions as Theorem \ref{thm:global}, there exists $C_z > 0$ such that 
    \begin{equation*}
        \norm{\bfZkn - \bfZk}\ff \leq C_z r^{(k)}, 
    \end{equation*}
    for any $k \in \bN$.
\end{lemma}

\begin{proof}
    It follows from the equality \eqref{eq:xk+1-xk} that
    \begin{equation*}
    \begin{aligned}
        \bfXkn - \bfXk
        = {} & (\bfW - I_{dn}) (\bfXk - \avXk) 
        - \eta \bfW (\bfHk - \bfE H (\barXk)) \\
        & - \eta \bfE (S (\barXk) - S (\cP (\barXk)))
        - \eta \bfE Q (\barXk)
        - \eta \bfE M \grad f (\cP(\barXk)),
    \end{aligned}
    \end{equation*}
    which further yields that
    \begin{equation*}
    %\begin{aligned}
        \norm{\bfXkn - \bfXk}\ff
        % \leq {} & 2 \norm{\bfXk - \avXk}\ff 
        % + \eta \norm{\bfHk - \bfE H (\barXk)}\ff
        % + \eta \sqrt{d} \norm{(S (\barXk) - S (\cP(\barXk)))} \\
        % & + \eta \sqrt{d} \norm{Q (\barXk)}\ff
        % + \eta \sqrt{d} \sigma_{\max} (M) \norm{ \grad f (\cP(\barXk))}\ff \\
        \leq \tilde{C}_z r^{(k)},
    %\end{aligned}
    \end{equation*}
    with $\tilde{C}_z := 2 + \sqrt{3 (C_2 + C_3 \beta^2)} + \eta \sqrt{d} (C_6 + \sqrt{7/6} \sigma_{\max}^{1 / 2} (M) \beta + \sigma_{\max} (M)) > 0$.
    Moreover, we have
    \begin{equation*}
        \norm{\bfUkn - \bfUk}\ff
        \leq 2 \norm{\bfUk - \avUk}\ff + L_c \norm{\bfXkn - \bfXk},
    \end{equation*}
    and
    \begin{equation*}
        \norm{\bfVkn - \bfVk}\ff
        \leq 2 \norm{\bfVk - \avVk}\ff + L_c \norm{\bfXkn - \bfXk}.
    \end{equation*}
    The above three inequalities imply that
    \begin{equation*}
        \norm{\bfZkn - \bfZk}\ff \leq \dkh{\dkh{1 + 4 L_c^2} \tilde{C}_z^2 + 16}^{1/2} r^{(k)},
    \end{equation*}
    which completes the proof with $C_z$ chosen as $((1 + 4 L_c^2) \tilde{C}_z^2 + 16)^{1/2}$.
\end{proof}

With Lemma \ref{le:nabla-h} and Lemma \ref{le:zk+1-zk}, we establish the convergence of the sequence $\{\bfXk\}$ generated by Algorithm \ref{alg:CDADT} when $\hbar$ is a K{\L} function, as stated in the following theorem.

\begin{theorem}
    Suppose that $\hbar$ is a K{\L} function.
    Then with the same conditions as Theorem \ref{thm:global}, there exists a first-order stationary point $X^{\ast} \in \SMnp$ of the problem \eqref{opt:dest} such that the sequence $\{\bfXk\}$ converges to $(\bfone_d \otimes I_n) X^{\ast}$. 
\end{theorem}

\begin{proof}
    According to Proposition \ref{prop:des-L}, the sequence $\{\hbar^{(k)}\}$ is nonincreasing and has a lower bound, which implies that the limit $\hbar^{\circ} := \lim_{k \to \infty} \hbar^{(k)}$ exists.
    Let $\Omega$ be the set of all the accumulation points of the sequence $\{\bfZk\}$.
    For any $\bfZ^{\circ} := (\bfX^{\circ}, \bfU^{\circ}, \bfV^{\circ}) \in \Omega$, we have 
    \begin{equation} \label{eq:limit-h}
        \hbar (\bfZ^{\circ}) = \lim_{k \to \infty} \hbar^{(k)} = \hbar^{\circ}.
    \end{equation}
    Thus, $\hbar$ is equal to the constant $\hbar^{\circ}$ on $\Omega$.
    If there exists $\bar{k} \in \bN$ such that $\hbar^{(\bar{k})} = \hbar^{\circ}$, the direct combination of Proposition \ref{prop:des-L} and Lemma \ref{le:zk+1-zk} would imply that $\bfZ^{(\bar{k} + 1)} = \bfZ^{(\bar{k})}$.
    Then a trivial induction shows that the assertion of this theorem is obvious.
    Since $\{\hbar^{(k)}\}$ is a nonincreasing sequence, it is clear from \eqref{eq:limit-h} that $\hbar^{(k)} > \hbar^{\circ}$.
    Moreover, for any $\tau > 0$, there exists $k^{\prime} \in \bN$ such that $\hbar^{(k)} < \hbar^{\circ} + \tau$ with $k > k^{\prime}$.
    According to Lemma 5 in \cite{Bolte2014proximal}, we know that $\Omega$ is a compact and connect set and
    \begin{equation*} %\label{eq:limit-dist}
        \lim_{k \to \infty} \dist (\bfZk, \Omega) = 0.
    \end{equation*}
    Hence, for any $\alpha > 0$, there exists $k^{\prime \prime}$ such that $\dist (\bfZk, \Omega) < \alpha$ with $k > k^{\prime\prime}$. 
    Summing up all these facts, we can obtain that $\bfZk$ belongs to the intersection of $\{\bfZ \mid \dist (\bfZk, \Omega) < \alpha\}$ and $\{\bfZ \mid \hbar^{\circ} < \hbar (\bfZ) < \hbar^{\circ} + \tau\}$ for any $k > \max\{k^{\prime}, k^{\prime\prime}\}$.
    By Lemma 6 in \cite{Bolte2014proximal}, there exists a constant $\tau > 0$ and a function $\phi \in \Phi_{\tau}$ such that
    \begin{equation*}
        \phi^{\prime} (\hbar^{(k)} - \hbar^{\circ}) \norm{\nabla \hbar (\bfZk)}\ff \geq 1,
    \end{equation*}
    for any $k > \max\{k^{\prime}, k^{\prime\prime}\}$.
    Multiplying both sides of \eqref{eq:des-L} by $\phi^{\prime} (\hbar^{(k)} - \hbar^{\circ})$ yields that
    \begin{equation*}
        C_e \phi^{\prime} (\hbar^{(k)} - \hbar^{\circ}) (r^{(k)})^2
        \leq \phi^{\prime} (\hbar^{(k)} - \hbar^{\circ}) (\hbar^{(k)} - \hbar^{(k + 1)})
        \leq \phi (\hbar^{(k)} - \hbar^{\circ}) - \phi (\hbar^{(k + 1)} - \hbar^{\circ}),
    \end{equation*}
    where the second inequality follows from the concavity of $\phi$.
    Combining the above relationship with Lemma \ref{le:nabla-h} and Lemma \ref{le:zk+1-zk}, we have
    \begin{equation*}
        \phi (\hbar^{(k)} - \hbar^{\circ}) - \phi (\hbar^{(k + 1)} - \hbar^{\circ})
        \geq \dfrac{C_e}{C_r} r^{(k)} \phi^{\prime} (\hbar^{(k)} - \hbar^{\circ}) \norm{\nabla \hbar (\bfZk)}\ff 
        %\geq \dfrac{C_e}{C_r} r^{(k)}
        \geq \dfrac{C_e}{C_r C_z} \norm{\bfZkn - \bfZk}\ff,
    \end{equation*}
    which further implies that
    \begin{equation*}
        \sum_{k = 0}^{s} \norm{\bfZkn - \bfZk}\ff \leq \dfrac{C_r C_z}{C_e} \dkh{\phi (\hbar^{(0)} - \hbar^{\circ}) - \phi (\hbar^{(s + 1)} - \hbar^{\circ})}
        \leq \dfrac{C_r C_z}{C_e} \phi (\hbar^{(0)} - \hbar^{\circ}),
    \end{equation*}
    for any $s \in \bN$.
    Letting $s \to \infty$, we can attain that
    \begin{equation*}
        \sum_{k = 0}^{\infty} \norm{\bfZkn - \bfZk}\ff < \infty.
    \end{equation*}
    Hence, the iterate sequence $\{\bfZk\}$ is a Cauchy sequence and hence is convergent,
    which infers that $\{\bfXk\}$ is also convergent.
    Finally, Theorem \ref{thm:global} guarantees that the limit point of $\{\bfXk\}$ has the form $(\bfone_d \otimes I_n) X^{\ast}$, where $X^{\ast} \in \SMnp$ is a first-order stationary point of the problem \eqref{opt:dest}. 
    This completes the proof.
\end{proof}

When $\hbar$ is a semialgebraic function, one important result is that the desingularizing function can be chosen to be of the form
\begin{equation*}
    \phi (t) = c t^{1 - \theta},
\end{equation*}
where $c > 0$ is a constant and $\theta \in [0, 1)$ is a parameter impacting the convergence rate.
Let $\bfZ^{\circ}$ be the limit point of the sequence $\bfZk$.
Using the same line of analysis introduced in \cite{Attouch2009convergence}, we can obtain the following estimations of convergence rates.
\begin{enumerate}

    \item If $\theta = 0$, the sequence $\bfZk$ converges in a finite number of steps.

    \item If $\theta \in (0, 1/2]$, there exists $\mu > 0$ and $\omega \in (0, 1)$ such that 
    \begin{equation*}
        \norm{\bfZk - \bfZ^{\circ}}\ff \leq \mu \omega^k.
    \end{equation*}

    \item If $\theta \in (1/2, 1)$, there exists $\mu > 0$ such that
    \begin{equation*}
        \norm{\bfZk - \bfZ^{\circ}}\ff \leq \mu k^{- \frac{1 - \theta}{2 \theta - 1}}.
    \end{equation*}
    
\end{enumerate}

\section{Numerical Experiments}

\label{sec:experiments}

In this section, we conduct a series of numerical experiments to demonstrate the efficiency and effectiveness of CDADT, specifically focusing on the CCA problems \eqref{opt:cca}.
The corresponding experiments are performed on a workstation with dual Intel Xeon Gold 6242R CPU processors (at $3.10$ GHz$\times 20 \times 2$) and $510$ GB of RAM under Ubuntu 20.04.
The tested algorithms are implemented in the $\mathtt{Python}$ language with the communication realized by the $\mathtt{mpi4py}$ package.

In the numerical experiments, the following three quantities are collected and recorded at each iteration as performance metrics.
\begin{itemize}
	
	\item Stationarity violation: $\norm{\barUk - \barVk \sym \dkh{(\barXk)\zz \barUk}}\ff$.
	
	\item Consensus error: $\sumiid \norm{\Xik - \barXk}\ff / d$.
	
	\item Feasibility violation: $\norm{(\barXk)\zz \barVk - I_p}\ff$.
	
\end{itemize}
Furthermore, we generate the Metropolis constant edge weight matrix \cite{Shi2015} as the mixing matrix for the tested networks.

\subsection{Numerical Results on Synthetic Datasets}

The first experiment is to evaluate the performance of CDADT on synthetic datasets. 
Specifically, in the CCA problem \eqref{opt:cca}, the first data matrix $A \in \Rnq$ (assuming $n \leq q$ without loss of generality) by its (economy-form) singular value decomposition as follows,
\begin{equation}\label{eq:gen-A}
	A = U S V\zz,
\end{equation}
where both $U \in \Rnn$ and $V \in \bR^{q \times n}$ are orthogonal matrices orthonormalized from randomly generated matrices, and $S \in \Rnn$ is a diagonal matrix with diagonal entries 
\begin{equation} \label{eq:Sigma_ii}
	S(i, i) = \xi_A^{i}, \quad \iin,
\end{equation}  
for a parameter $\xi_A \in (0, 1)$ that determines the decay rate of the singular values of $A$.
The second data matrix $B \in \Rmq$ is generated in a similar manner but with a different decay rate $\xi_B \in (0, 1)$ of the singular values.
After construction, the columns of the data matrices $A$ and $B$ are uniformly distributed into $d$ agents.

We test the performances of CDADT with different choices of penalty parameters on the Erd{\"o}s-R{\'e}nyi (ER) network.
The data matrices $A \in \Rnq$ and $B \in \Rmq$ are randomly generated with $n = 20$, $m = 30$, $q = 3200$, $\xi_A = 0.97$, and $\xi_B = 0.96$. 
And the CCA problem \eqref{opt:cca} is tested with $p = 5$ and $d = 32$.
The corresponding numerical results are provided in Figure \ref{fig:CCA_beta}, which presents the performances of CDADT with $\beta \in \{0.01, 0.1, 1, 10, 100\}$.
It can be observed that the curves of three performance metrics almost coincide with each other, which corroborates the robustness of CDADT to the penalty parameter in a wide range.

\begin{figure}[ht!]
	\centering
	
	\subfigure[Stationarity Violation]{
		\label{subfig:CCA_beta_kkt}
		\includegraphics[width=0.3\linewidth]{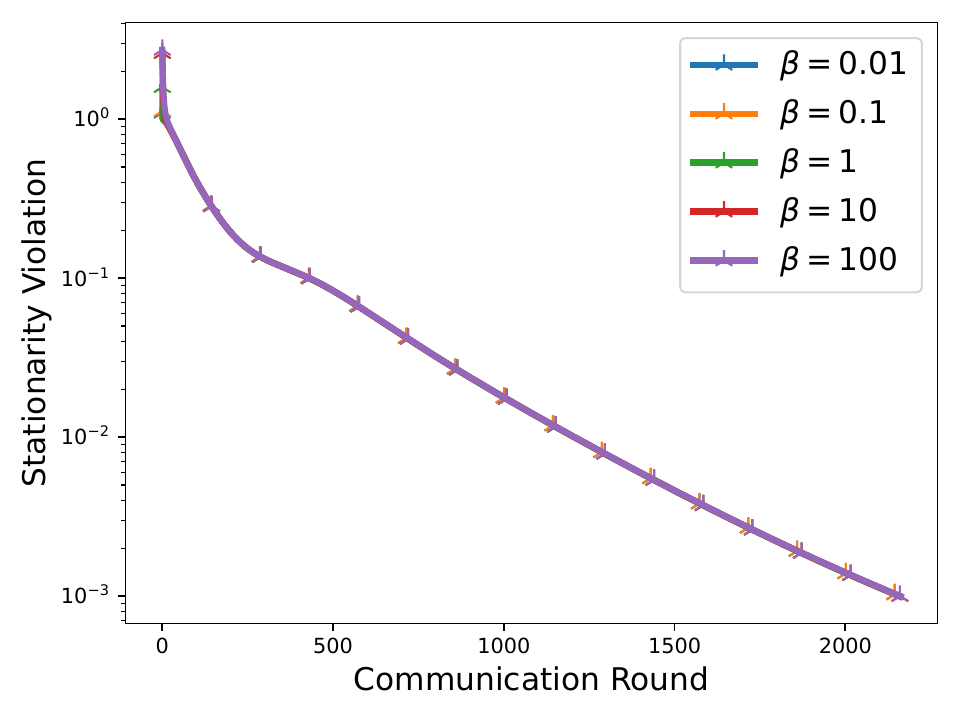}
	}
	\subfigure[Consensus Error]{
		\label{subfig:CCA_beta_cons}
		\includegraphics[width=0.3\linewidth]{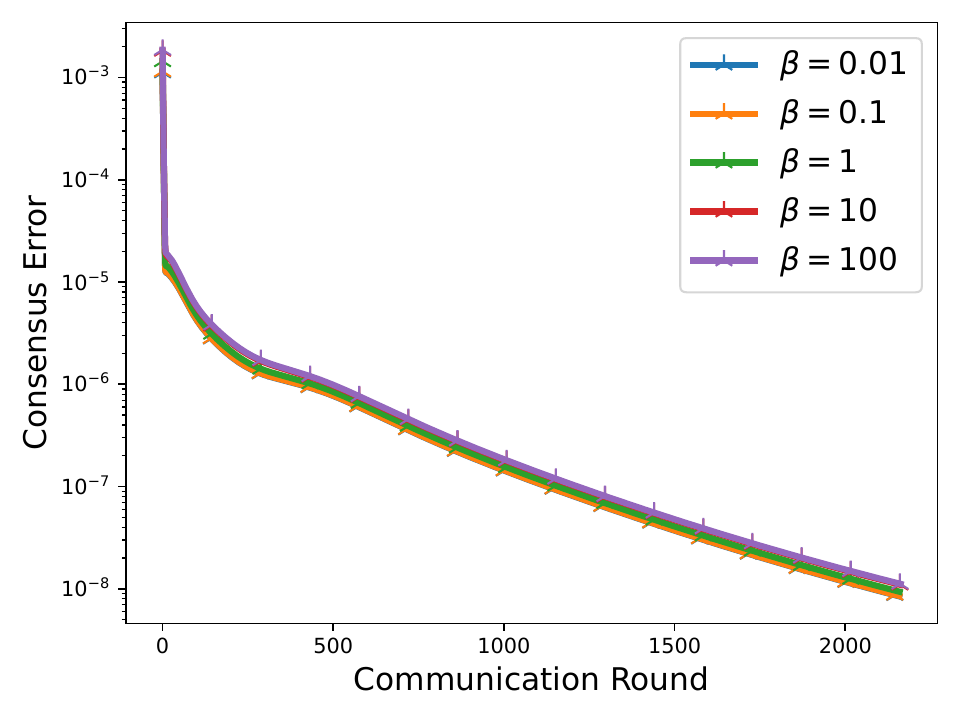}
	}
	\subfigure[Feasibility Violation]{
		\label{subfig:CCA_beta_feas}
		\includegraphics[width=0.3\linewidth]{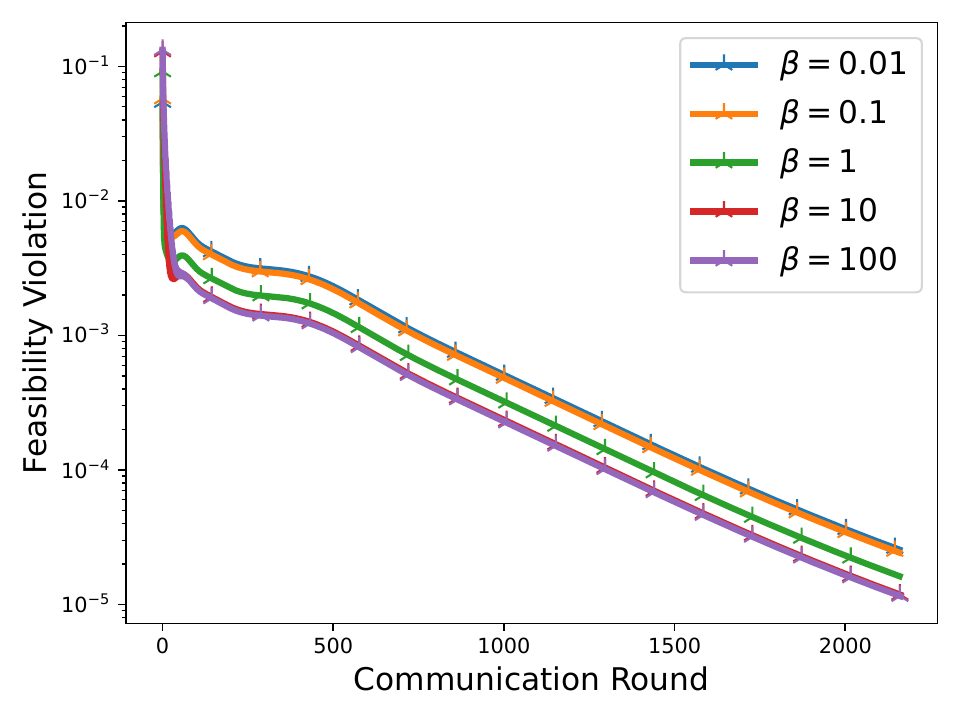}
	}
	
	\caption{Numerical performances of CDADT for different values of $\beta$.}
	\label{fig:CCA_beta}
\end{figure}

Furthermore, we conduct numerical tests to evaluate the impact of network topologies on the performance of CDADT, focusing on ring networks, grid networks, and ER networks. 
Figure \ref{fig:network} illustrates the structures of these networks and the corresponding values of $\lambda$ defined in \eqref{eq:lambda}.
For our experiment, the data matrices $A \in \Rnq$ and $B \in \Rmq$ are randomly generated with $n = m = 50$, $q = 3200$, $\xi_A = 0.99$, and $\xi_B = 0.98$. 
Then the CCA problem \eqref{opt:cca} is tested with $p = 5$ and $d = 16$.
We set the algorithmic parameters $\eta = 0.0001$ and $\beta = 1$ in CDADT.
Figure \ref{fig:CCA_network} depicts the diminishing trend of three performance metrics against the communication rounds on a logarithmic scale, with different networks distinguished by colors.
We can observe that, as the network connectivity becomes worse (i.e., $\lambda$ approaches $1$), our algorithm requires more communication rounds to achieve the specified accuracy.

\begin{figure}[ht!]
	\centering
	
	\subfigure[ER $(\lambda \approx 0.82)$]{
		\label{subfig:er}
		\includegraphics[width=0.25\linewidth]{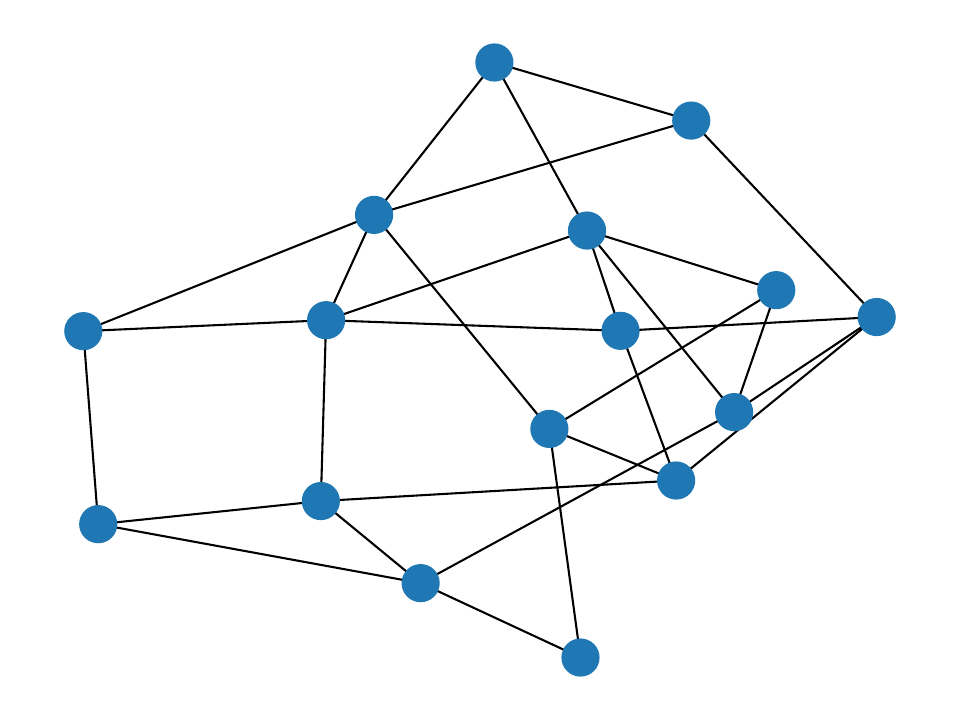}
	}
	\subfigure[Grid $(\lambda \approx 0.87)$]{
		\label{subfig:grid}
		\includegraphics[width=0.25\linewidth]{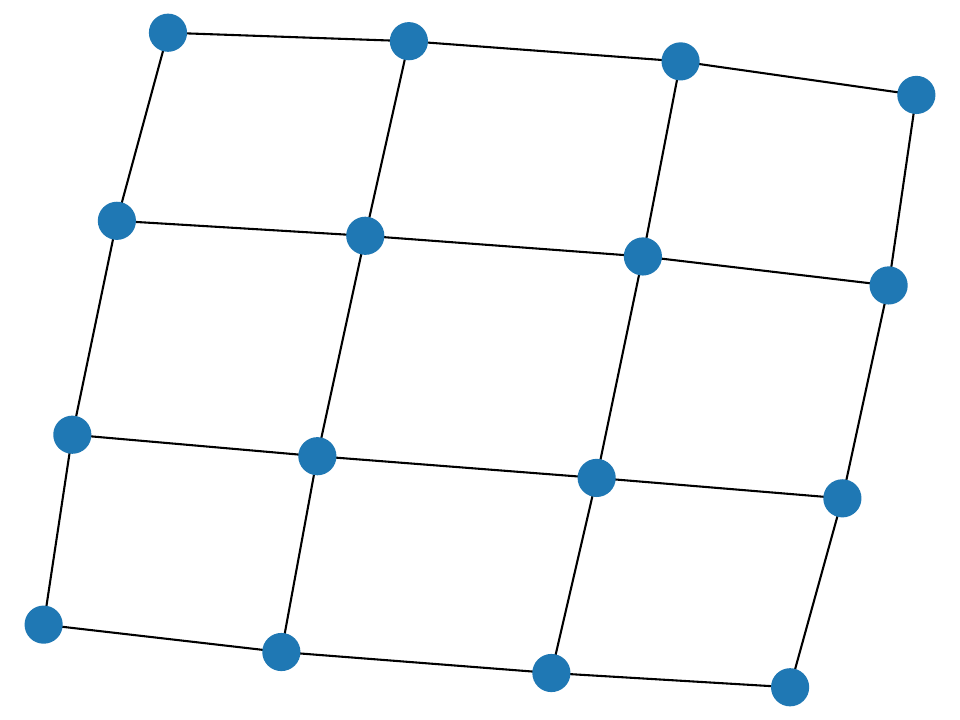}
	}
	\subfigure[Ring $(\lambda \approx 0.95)$]{
		\label{subfig:ring}
		\includegraphics[width=0.25\linewidth]{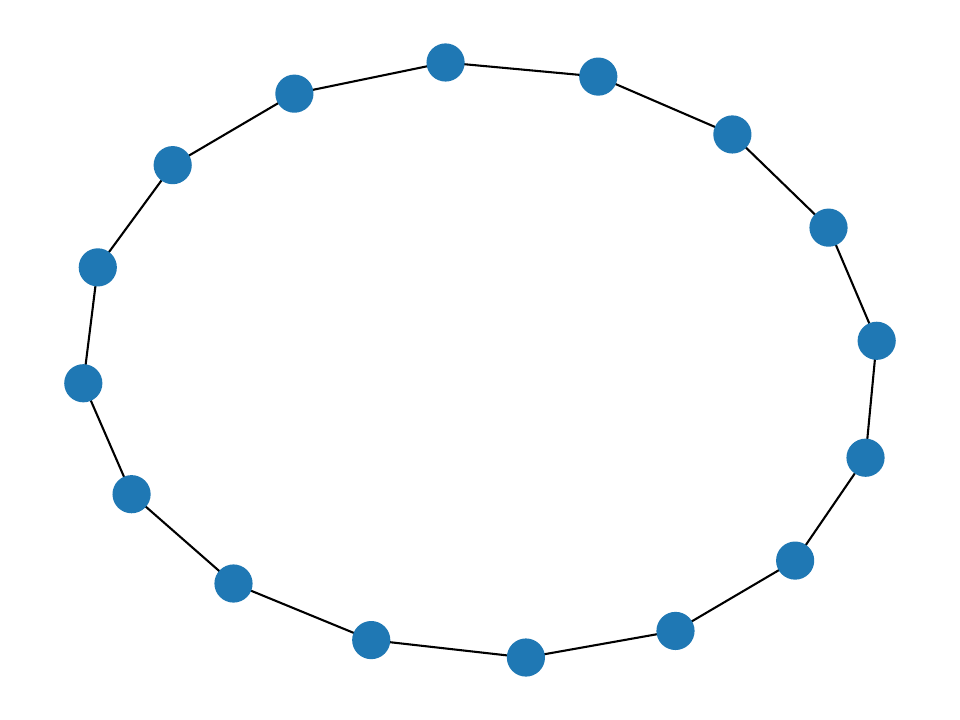}
	}
	
	\caption{Illustration of different network structures.}
	\label{fig:network}
\end{figure}

\begin{figure}[ht!]
	\centering
	
	\subfigure[Stationarity Violation]{
		\label{subfig:CCA_kkt_network}
		\includegraphics[width=0.3\linewidth]{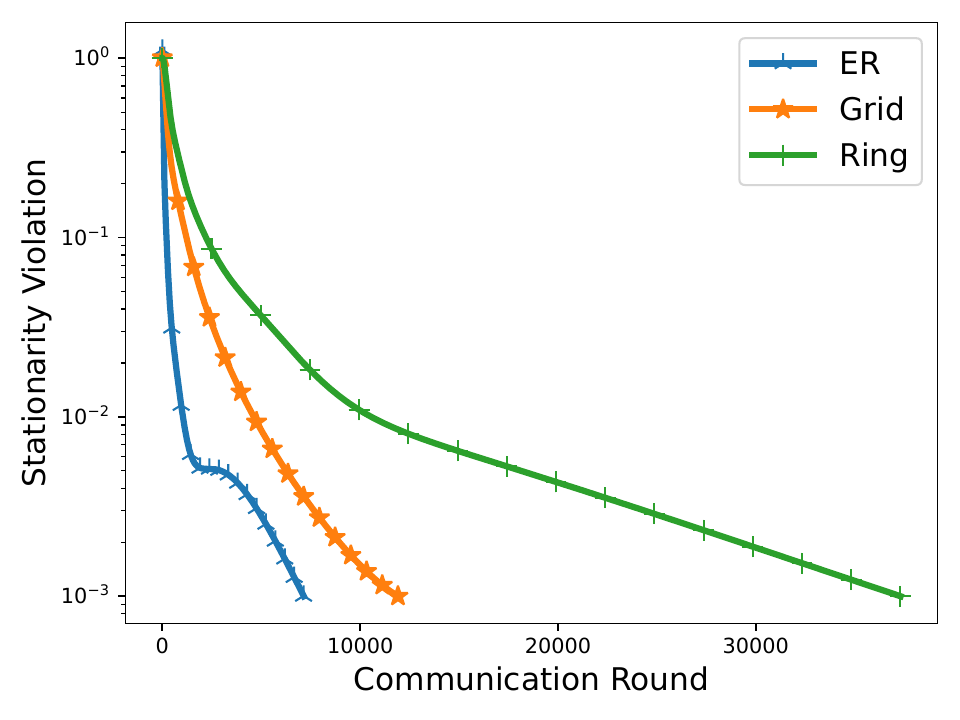}
	}
	\subfigure[Consensus Error]{
		\label{subfig:CCA_cons_network}
		\includegraphics[width=0.3\linewidth]{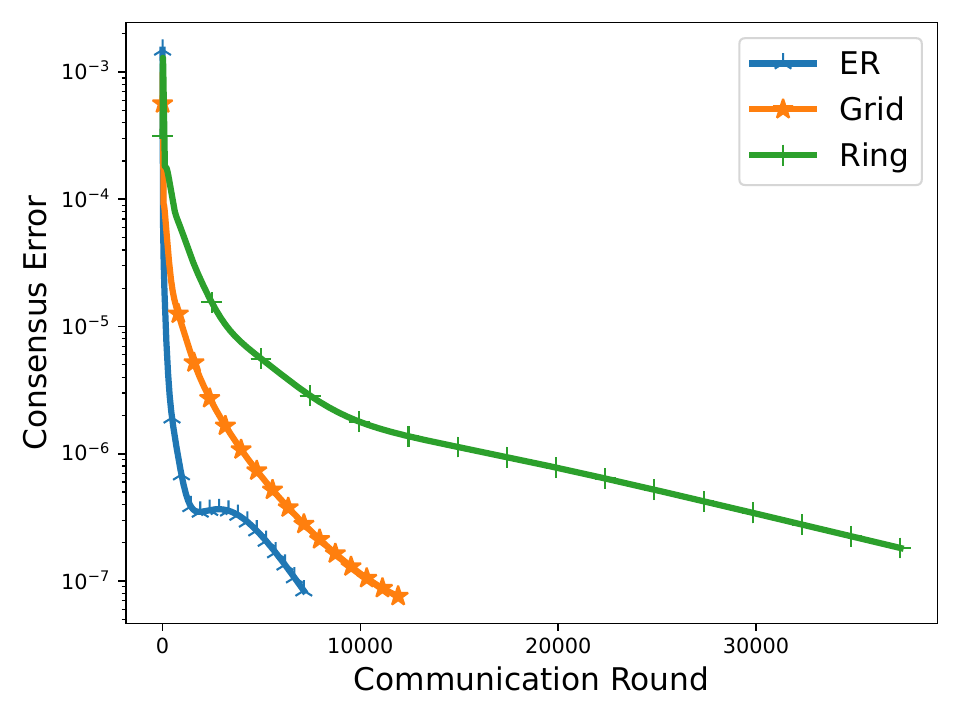}
	}
	\subfigure[Feasibility Violation]{
		\label{subfig:CCA_feas_network}
		\includegraphics[width=0.3\linewidth]{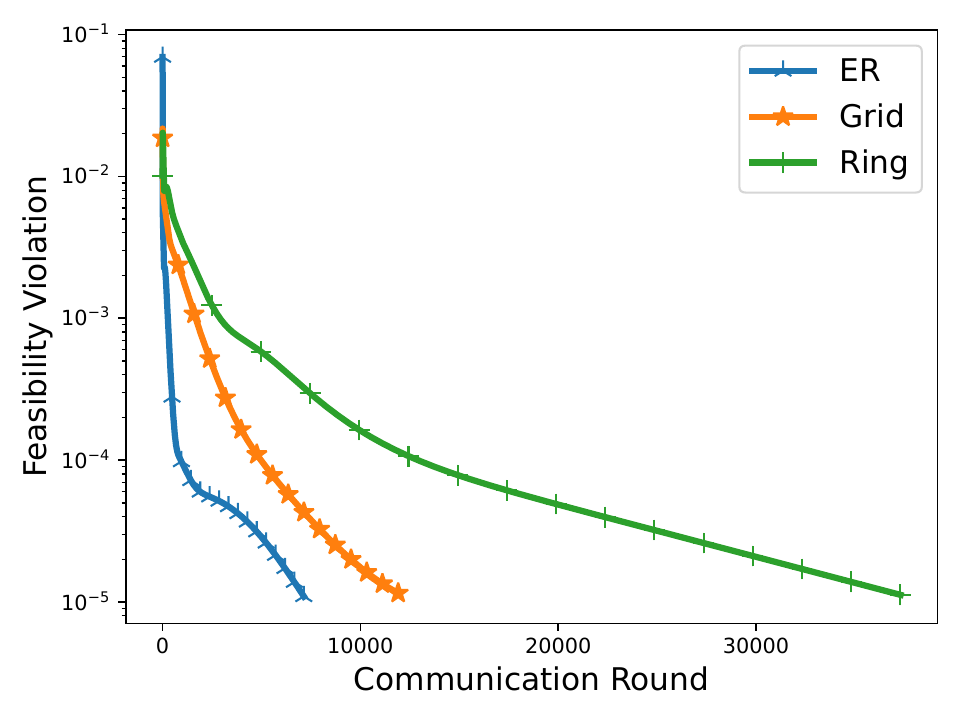}
	}
	
	\caption{Numerical performances of CDADT on synthetic datasets across three different networks.}
	\label{fig:CCA_network}
\end{figure}

\subsection{Numerical Results on Real-world Datasets}

Next, we engage in a numerical test to assess the effectiveness of CDADT on two real-world datasets, including MNIST \cite{LeCun1998gradient} and Mediamill \cite{Snoek2006challenge}.
Specifically, MNIST is a database of handwritten digits, consisting of gray-scale images of size $28 \times 28$. 
Every image is split into left and right halves, which are used as the two views with $n = m = 392$.
We employ CCA to learn the correlated representations between left and right halves of the images, which involves the full training set of MNIST containing $q = 60000$ images.
In the Mediamill dataset, each image is a representative keyframe of a video shot containing $n = 120$ features, which is annotated with $m = 101$ labels.
We extract the first $q = 43200$ samples to test the CCA problem, which is performed to explore the correlation structure between images and labels.

For our testing, we employ CDADT to solve the CCA problem \eqref{opt:cca} on the ER network, where we fix $p = 5$ and $d = 32$.
And the algorithmic parameters are set to $\eta = 0.005$ and $\beta = 1$ in CDADT.
The corresponding numerical results are presented in Figure \ref{fig:CCA_real}.
It is noteworthy that the effectiveness of CDADT is not limited to synthetic datasets but also extends to real-world applications.
Moreover, CDADT demonstrates its potential to solve CCA problems over large-scale networks.

\begin{figure}[ht!]
	\centering
	
	\subfigure[MNIST]{
		\label{subfig:CCA_real_MNIST}
		\includegraphics[width=0.4\linewidth]{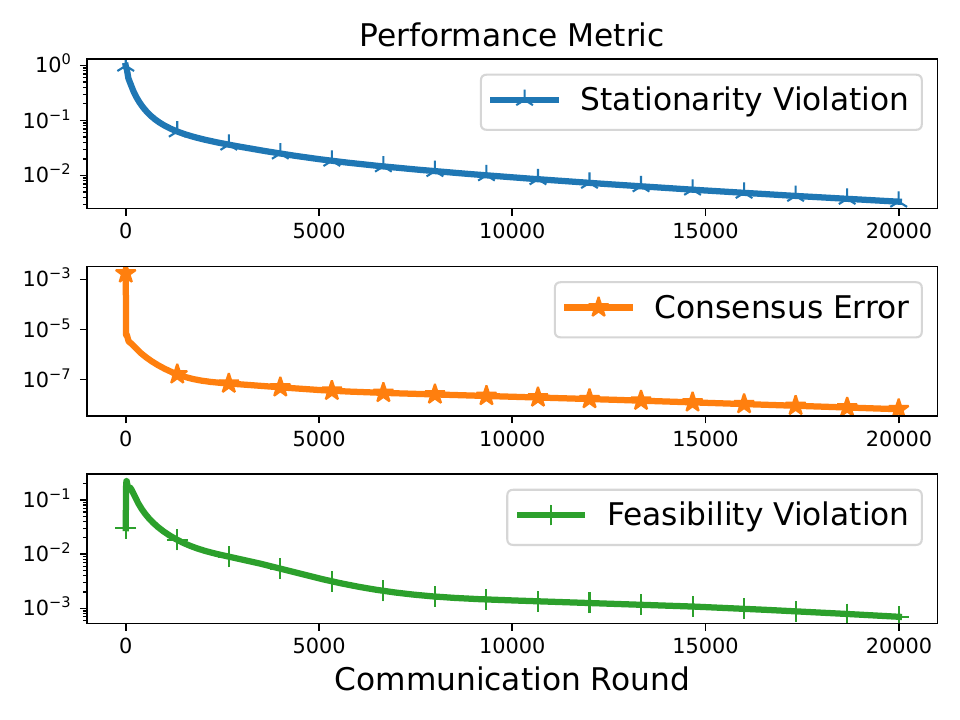}
	}
	\subfigure[Mediamill]{
		\label{subfig:CCA_real_Mediamill}
		\includegraphics[width=0.4\linewidth]{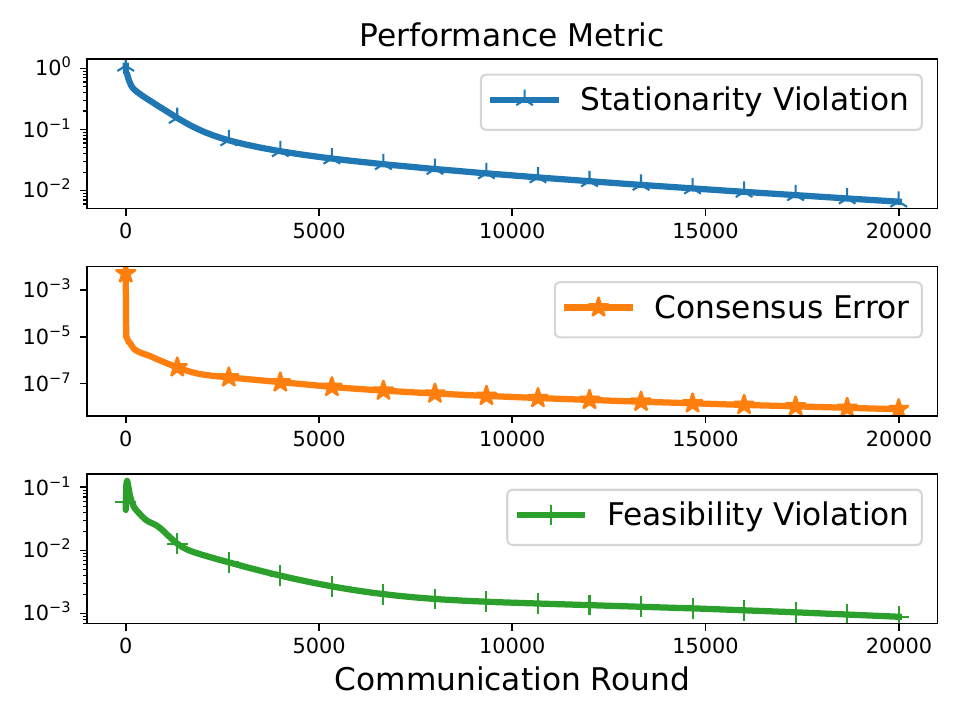}
	}
	
	\caption{Numerical performances of CDADT on two real-world datasets.}
	\label{fig:CCA_real}
\end{figure}

\section{Conclusion}

\label{sec:conclusion}

Decentralized optimization problems with generalized orthogonality constraints arise in various scientific and engineering applications.
Existing algorithms for solving Riemannian optimization problems heavily rely on geometric tools of the underlying Riemannian manifold, such as tangent spaces and retraction operators.
However, these algorithms can not be applied to solve \ref{opt:cca-dist}, where the manifold constraints possess an inherently distributed structure. To surmount this intricate challenge, we propose to employ the constraint dissolving operator to build up an exact penalty model of the original problem.
Nevertheless, existing algorithms remain unsuitable for solving the newly derived penalty model, as the penalty function is not entirely separable over that agents.
To address these challenges, we develop an efficient decentralized algorithm based on the double-tracking strategy.
In order to construct a descent direction of the penalty function, our algorithm not only tracks the gradient of the objective function but also maintains a global estimate of the Jacobian of the constraint mapping.
The proposed algorithm is guaranteed to converge to a first-order stationary point under mild conditions.
We validate its performance through numerical experiments conducted on both synthetic and real-world datasets.

As for future works, we are interested in extending our algorithm to general constraints such that it can find a wider range of applications. Moreover, it is worthy of investigating the performance of CDADT in stochastic and online settings.

\paragraph{Acknowledgement}

The third author was supported in part by the National Natural Science Foundation of China (12125108, 12226008, 11991021, 11991020, 12021001, 12288201), Key Research Program of Frontier Sciences, Chinese Academy of Sciences (ZDBS-LY-7022), and CAS--Croucher Funding Scheme for Joint Laboratories ``CAS AMSS--PolyU Joint Laboratory of Applied Mathematics: Nonlinear Optimization Theory, Algorithms and Applications''.

% ---------------------------------------------------------------------------------------------------------------------------------

\bibliographystyle{plainnat}

\bibliography{library}

\addcontentsline{toc}{section}{References}

\end{document}